\documentclass[pdflatex,sn-mathphys-num]{sn-jnl}% Math and Physical Sciences Numbered Reference Style
\usepackage{graphicx}%
\usepackage{multirow}%
\usepackage{amsmath,amssymb,amsfonts}%
\usepackage{amsthm}%
\usepackage{mathrsfs}%
\usepackage[title]{appendix}%
\usepackage{xcolor}%
\usepackage{textcomp}%
\usepackage{manyfoot}%
\usepackage{booktabs}%
\usepackage{algorithm}%
\usepackage{algorithmicx}%
\usepackage{algpseudocode}%
\usepackage{listings}%

\usepackage{bm}
\usepackage{mathtools}
\newcommand{\q}{\quad}

\theoremstyle{thmstyleone}%
\newtheorem{theorem}{Theorem}%  meant for continuous numbers
\theoremstyle{thmstyletwo}%
\newtheorem{remark}{Remark}%
\newtheorem{lemma}{Lemma}%
\newtheorem{assumption}{Assumption}%

\theoremstyle{thmstylethree}%

\begin{document}

\title[Article Title]{On solving nonlinear simultaneous equations arising from the double-exponential Sinc-collocation method for initial value problems}


\author*[1]{\fnm{Yusaku} \sur{Yamamoto}}\email{yusaku.yamamoto@uec.ac.jp}
\author[2]{\fnm{Ken'ichiro} \sur{Tanaka}}\email{kenichiro@comp.isct.ac.jp}
\equalcont{These authors contributed equally to this work.}

\affil*[1]{\orgdiv{Department of Computer and Network Engineering, Graduate School of Informatics and Engineering}, \orgname{The University of Electro-Communications}, \orgaddress{\street{1-5-1 Chofugaoka}, \city{Chofu}, \state{Tokyo}, \postcode{182-8585}, \country{Japan}}}
\affil[2]{\orgdiv{Department of Mathematical and Computing Science, School of Computing}, \orgname{Institute of Science Tokyo}, \orgaddress{\street{2-12-1 Ookayama}, \city{Meguro-ku}, \state{Tokyo}, \postcode{152-8550}, \country{Japan}}}


\abstract{The double-exponential Sinc-collocation method is known as a super-accurate method for solving initial value problems of ordinary differential equations, for which the error decreases almost exponentially as a function of the number of sample points in the temporal direction, $N$. However, this method requires solving nonlinear simultaneous equations in $nN$ variables when the problem dimension is $n$. Recently, Ogata pointed out that Gauss-Seidel type fixed-point iteration works surprisingly well for solving these equations, typically reducing the error by one or two orders of magnitude at each iteration. In this paper, we analyze the convergence of this iteration and give a sufficient condition for its global convergence. We also provide an upper bound on its convergence factor, which explains the efficiency of this iteration. Some numerical examples that illustrate the validity of our analysis are also provided.}




\maketitle

\section{Introduction}
In this paper, we consider numerical solution of the initial value problem
\begin{equation}
\frac{{\rm d}{\bm x}}{{\rm d}t}={\bm f}(t,{\bm x}) \quad (a\le t\le b), \quad {\bm x}(a)={\bm x}_a \in\mathbb{R}^n,
\label{eq:ODE}
\end{equation}
where ${\bm f}(t,{\bm x})\in\mathbb{R}^{n+1}\rightarrow\mathbb{R}^n$. For this problem, several super-accurate quadrature-based numerical methods, for which the error decays nearly exponentially as a function of the number of sample points $N$ in the interval $[a,b]$, have been proposed \cite{Tanaka23,Stenger93,Nurmuhammad05,Okayama18,Ogata23}. For example, the method of Stenger \cite{Stenger93}, which combines the Sinc-collocation method with the SE (single-exponential) transformation, achieves convergence rate of $O(\sqrt{N}\exp(-c\sqrt{N}))$, where $c$ is some constant. Nurmuhammad et al.~\cite{Nurmuhammad05} propose to combine the Sinc-Nystr\"{o}m method with the SE or DE (double-exponential) transformation. Theoretical analysis of these methods is given by Okayama \cite{Okayama18} for the special case where ${\bm f}(t,{\bm x})$ is a linear function in ${\bm x}$ and it is shown that the convergence rate is $O(\exp(-c\sqrt{N}))$ and $O(\frac{\log N}{N}\exp(-cN/\log N))$ for the SE and DE case, respectively. Note that the value of $c$ differs for each method. Okayama \cite{Okayama18} also proposes to replace the SE transformation in Stenger's method with the DE transformation and show that the convergence rate of the resulting method, the DE Sinc-collocation method, is $O(\exp(-cN/\log N))$. More recently, Ogata \cite{Ogata23} proposes to combine the periodic Sinc approximation with the IMT-type DE transformation. While the convergence of the resulting method is slightly slower ($O(N^{3/2}\log N\exp(-cN/\log^2 N))$) than those of the DE Sinc-Nystr\"{o}m and DE Sinc-collocation methods, it has the advantage that truncation of the integral domain required by the DE transformation is not needed. Thus, the parameter to specify truncation is eliminated and the user needs to specify only one parameter, $N$.

All of these methods reduce the initial value problem \eqref{eq:ODE} to a system of nonlinear equations in the values of ${\bm x}$ at $N$ sample points. Hence, nonlinear simultaneous equations in $nN$ variables have to be solved and this usually accounts for most of the computational work. Several approaches such as the Newton method and Gaussian elimination (in the linear case) \cite{Nurmuhammad05} can be used to solve these equations. Among them, Ogata \cite{Ogata23} proposes to use fixed-point iteration, which corresponds to a discrete version of the Picard iteration. There are several variants of this iteration which differ in the way of updating the variables, and he reports that the Gauss-Seidel type variant works surprisingly well, reducing the error by one or two orders of magnitude at each iteration for the test problems. While this iteration was devised for the IMT-DE Sinc-collocation method, it can be applied to other numerical methods mentioned above as well.

In this paper, we analyze the convergence of this Gauss-Seidel type fixed-point iteration. As a target numerical method, we choose the DE Sinc-collocation method, because it is conceptually simple and easier to analyze. We assume that ${\bm f}(t,{\bm x})$ is Lipschitz continuous in ${\bm x}$ and derive a bound on the convergence rate in terms of $N$, the interval length $b-a$ and the Lipschitz constant $L$. Contrary to the ordinary behavior of the Gauss-Seidel method, it is shown that the convergence gets faster as $N$ increases. This theoretical prediction is confirmed by numerical experiments.

The rest of this paper is structured as follows. In Section \ref{DE_Sinc}, we describe the DE Sinc-collocation method for initial value problems and the Gauss-Seidel type fixed-point iteration to solve the nonlinear simultaneous equations arising from it. Section \ref{Analysis} provides theoretical convergence analysis of the Gauss-Seidel type iteration. Numerical results that support our analysis are given in Section \ref{Results}. Finally Section \ref{Conclusion} gives some conclusion.

Throughout this paper, scalars, vectors and matrices are denoted by small letters (either Roman or Greek), bold italic letters and capital letters, respectively. $\|\cdot\|_{\infty}$ denotes the infinity-norm of a vector or a matrix. For a matrix $A=(a_{ij})$, $|A|$ denotes a matrix whose elements are $|a_{ij}|$. Inequalities between matrices mean elementwise inequalities. For $x>0$, $\log x$ means the natural logarithm of $x$.

\section{The DE Sinc-collocation method for initial value problems}
\label{DE_Sinc}
\subsection{Sinc-collocation methods with variable transformation}
Let us consider the initial value problem \eqref{eq:ODE}. We begin with describing a general procedure to solve this equation by the (periodic or non-periodic) Sinc-collocation methods with variable transformation. The first step is to rewrite it as an equivalent Volterra integral equation of the second kind:
\begin{equation}
{\bm x}(t)={\bm x}_a+\int_a^t{\bm f}(\tau,{\bm x}(\tau))\,{\rm d}\tau \quad (a\le t\le b).
\label{eq:integral_equation}
\end{equation}
Next, using some monotonically increasing function $\varphi$, we apply change of variables $\tau=\varphi(s)$ to the integral in \eqref{eq:integral_equation}, obtaining
\begin{equation}
{\bm x}(t)={\bm x}_a+\int_{\varphi^{-1}(a)}^{\varphi^{-1}(t)}{\bm f}(\varphi(s),{\bm x}(\varphi(s)))\varphi'(s)\,{\rm d}s \quad (a\le t\le b).
\label{eq:integral_equation2}
\end{equation}
Now, by using the function values at some sample points $s_{N_1}<s_{N_1+1}<\ldots<s_{N_2}$ on the $s$ axis, we interpolate the integrand in \eqref{eq:integral_equation2} as
\begin{equation}
{\bm f}(\varphi(s),{\bm x}(\varphi(s)))\varphi'(s) \simeq \sum_{j=N_1}^{N_2} {\bm f}(\varphi(s_j),{\bm x}(\varphi(s_j)))\varphi'(s_j)p_j(s),
\label{eq:interpolation}
\end{equation}
where $\{p_j(s)\}_{j=N_1}^{N_2}$ are some interpolation functions that satisfy $p_j(s_i)=\delta_{ij}$. Note that the function values ${\bm x}(\varphi(s_j))$ are unknown at this stage. Substituting \eqref{eq:interpolation} into \eqref{eq:integral_equation} and defining $P_j(s)$ by
\begin{equation}
P_j(s) \equiv \int_{\varphi^{-1}(a)}^s p_j(s')\,{\rm d}s',
\label{eq:pintegral}
\end{equation}
we can perform the integration to get
\begin{equation}
\int_{\varphi^{-1}(a)}^{\varphi^{-1}(t)}{\bm f}(\varphi(s),{\bm x}(\varphi(s)))\varphi'(s)\,{\rm d}s \simeq \sum_{j=N_1}^{N_2}{\bm f}(\varphi(s_j), {\bm x}(\varphi(s_j)))\varphi'(s_j)P_j(\varphi^{-1}(t)) \quad (a\le t\le b).
\label{eq:integral_equation3}
\end{equation}
Now, we insert the approximation \eqref{eq:integral_equation3} into \eqref{eq:integral_equation2} and equate the both sides at $N_2-N_1+1$ points $t=\varphi(s_1), \ldots, \varphi(s_M)$ (the collocation method). Then, rewriting ${\bm x}(\varphi(s_i))$ as $\widetilde{\bm x}_i$ ($i=N_1, N_1+1, \ldots, N_2$), we obtain the following nonlinear simultaneous equations in $\widetilde{\bm x}_{N_1}, \widetilde{\bm x}_{N_1+1}, \ldots \widetilde{\bm x}_{N_2}$.
\begin{equation}
\widetilde{\bm x}_i = {\bm x}_a + \sum_{j=N_1}^{N_2} {\bm f}(\varphi(s_j),\widetilde{\bm x}_j)\varphi'(s_j)P_j(s_i) \quad (i=N_1,\ldots,N_2).
\label{eq:nonlinear_equation}
\end{equation}
Solving these equations and substituting the results into \eqref{eq:integral_equation3} and \eqref{eq:integral_equation2}, we obtain the following approximate solution of \eqref{eq:ODE}.
\begin{equation}
\widetilde{\bm x}(t)={\bm x}_a+\sum_{j=N_1}^{N_2}{\bm f}(\varphi(s_j), \widetilde{\bm x}_j)\varphi'(s_j)P_j(\varphi^{-1}(t)) \quad (a\le t\le b).
\label{eq:solution}
\end{equation}
Note that the approximate solution is given as a function of $t$, which is available at any point in $[a,b]$. This is one of the advantages of this type of numerical methods.

As examples of numerical methods based on this procedure, we can list the DE Sinc-collocation method by Okayama \cite{Okayama18} and the IMT-DE periodic Sinc-collocation method by Ogata \cite{Ogata23}. The former method uses the double-exponential transformation for $\varphi(s)$ and the sample points $s_j=jh$ ($j=N_{1},\ldots,N_{2}$), where $h > 0$ is the step size, and combines it with the Sinc interpolation formula. Thus,
\begin{align}
\varphi(s) &= \varphi_{\rm DE}(s) \equiv \frac{b-a}{2}\tanh\left(\frac{\pi}{2}\sinh(s)\right)+\frac{b+a}{2}, \label{eq:DE} \\
p_j(s) &= S(j,h)(s) \equiv {\rm sinc}\left(\frac{s-jh}{h}\right)\quad (j=N_{1},\ldots,N_{2}), \label{eq:sinc}
\end{align}
where
\begin{equation}
{\rm sinc}(x)=\begin{cases}
\sin(\pi x)/(\pi x), & x\ne 0, \\
1 & x=0.
\end{cases}
\end{equation}
In this case, $P_j(s)$ is given as
\begin{equation}
P_j(s) = J(j,h)(s) \equiv h\left\{\frac{1}{2}+\frac{1}{\pi}{\rm Si}\left(\frac{\pi(s-jh)}{h}\right)\right\}, \label{eq:Pjs}
\end{equation}
where ${\rm Si}(x)=\int_0^x(\sin t/t){\rm d}t$ is the sine integral. In the latter method, the IMT-DE type transformation is used as $\varphi(s)$ and the sample points are chosen as $s_j=jh$ ($j=0, 1, \ldots, 2N$, $h=1/(2N+1)$), 
where $N$ is a positive integer%
\footnote{In this case, we set $N_{1} = 0$ and $N_{2} = 2N$.}. 
For interpolation, the periodic Sinc function is used. Thus,

\begin{align}
\varphi(s) &= \varphi_{\rm IMT-DE}(s) \equiv \frac{b-a}{2}\tanh\left(A\sinh B\left(\frac{1}{1-s}-\frac{1}{s}\right)\right)+\frac{b+a}{2}, \label{eq:IMT-DE} \\
p_j(s) &= S^{({\rm p})}(j,h)(s) \equiv h\frac{\sin\left(\frac{\pi}{h}(s-jh)\right)}{\sin\left(\pi(s-jh)\right)} \quad (j=0, 1, \ldots, 2N). \label{eq:psinc}
\end{align}
For this method, $P_j(s)$ is given as
\begin{equation}
P_j^{({\rm p})}(s) = J^{({\rm p})}(j,h)(s) \equiv h\left\{s+\sum_{n=1}^N\frac{1}{n\pi}\left(\sin(2n\pi(s-jh))+\sin(2n\pi jh)\right)\right\}. \label{eq:Pjps}
\end{equation}
For these methods, it has been shown that under some technical conditions on ${\bm f}(t,{\bm x})$, the maximum error of the approximate solution $\widetilde{\bm x}(t)$ in the interval $a\le t\le b$ decreases almost exponentially with respect to 
%$N$ 
the number of the sample points \cite{Okayama18,Ogata23}.

\begin{remark}
In the above explanation, we constructed the approximate solution $\widetilde{\bm x}(t)$ for arbitrary $t$ by inserting $\{\widetilde{\bm x}_i\}_{i=N_1}^{N_2}$ into \eqref{eq:integral_equation3} and \eqref{eq:integral_equation2}. Alternatively, we can construct $\widetilde{\bm x}(t)$ by interpolating $\{\widetilde{\bm x}_i\}_{i=N_1}^{N_2}$ using the Sinc functions in the $s$-space. Strictly speaking, the former method is called the DE Sinc-Nystr\"{o}m method, while the latter is called the DE Sinc-collocation method \cite{Okayama18}. Both methods solve the same set of nonlinear equations \eqref{eq:nonlinear_equation}, so our analysis in the next section is applicable to both of them.
\end{remark}

\subsection{Solution of the nonlinear simultaneous equations}
In the numerical methods outlined in the previous subsection, most of the computation time is spent for solving the nonlinear simultaneous equations \eqref{eq:nonlinear_equation}. Hence, their efficient solution is of great importance. Several algorithms can be considered to solve \eqref{eq:nonlinear_equation}, such as the Newton method and fixed-point iteration. Among them, Ogata proposes to use a Jacobi type iteration and Gauss-Seidel type iteration \cite{Ogata23}. While he applies these iterations to the IMT-DE periodic Sinc-collocation method, they can be used for the DE Sinc-collocation method as well. Let the approximate solution of \eqref{eq:nonlinear_equation} at the $\nu$th iteration be denoted by $\widetilde{\bm x}_i^{(\nu)}$ and let the initial values be $\widetilde{\bm x}_i^{(0)}={\bm x}_a$ for $i=N_1,\ldots,N_2$. Also, let
\begin{equation}
w_{ij}\equiv\varphi'(s_j)P_j(s_i).
\label{eq:wij}
\end{equation}
Then, the recursion formula for the Jacobi type iteration reads
\begin{equation}
\widetilde{\bm x}_i^{(\nu+1)} = {\bm x}_a + \sum_{j=N_1}^{N_2} w_{ij}{\bm f}(\varphi(s_j),\widetilde{\bm x}_j^{(\nu)}) \quad (i=N_1,\ldots,N_2).
\label{eq:Jacobi}
\end{equation}
On the other hand, the Gauss-Sidel type iteration can be written as follows.
\begin{equation}
\widetilde{\bm x}_i^{(\nu+1)} = {\bm x}_a + \sum_{j=N_1}^{i-1} w_{ij}{\bm f}(\varphi(s_j),\widetilde{\bm x}_j^{(\nu+1)}) + \sum_{j=i}^{N_2} w_{ij}{\bm f}(\varphi(s_j),\widetilde{\bm x}_j^{(\nu)}) \quad (i=N_1,\ldots,N_2). \label{eq:Gauss-Seidel}
\end{equation}

Ogata points out that the Jacobi type iteration \eqref{eq:Jacobi} can be interpreted as a discrete version of the Picard iteration to solve the integral equation \eqref{eq:integral_equation}. Based on this observation, he analyzes the convergence of this iteration applied to the IMT-DE periodic Sinc-collocation method and concludes that the iterate $\widetilde{\bm x}_i^{(\nu)}$ converges almost superlinearly to the solution of \eqref{eq:nonlinear_equation} \cite{Ogata23}. For the Gauss-Seidel type iteration \eqref{eq:Gauss-Seidel}, Ogata does not provide an analysis, but reports that the convergence is extremely fast for the test problems, reducing the error by nearly two orders of magnitude at each iteration \cite{Ogata23}. In the next section, we will analyze the convergence of this Gauss-Seidel type iteration theoretically, focusing mainly on the DE Sinc-collocation method.

\section{Convergence analysis of the Gauss-Seidel type iteration for the DE Sinc-collocation method}
\label{Analysis}
It is well known in the context of iterative solution of linear systems that the Gauss-Seidel method is generally faster than the Jacobi method. For example, for a linear system arising from discretization of the 2-dimensional Poisson equation by the finite-difference method using 5-point stencil, if the coefficient matrix is consistently ordered, the convergence of the Gauss-Seidel method is twice as fast as that of the Jacobi method. On the other hand, in the present problem, the coefficient matrix $(w_{ij})$ is dense, so the above argument does not apply. In addition, the equation to be solved is generally nonlinear. Yet, the numerical results given in \cite{Ogata23} shows that the Gauss-Seidel type iteration is much more efficient than the Jacobi type iteration, requiring less than half the number of iterations of the latter. Hence, it should be of interest to elucidate the reason of this efficiency.

To start the analysis, we define the vector $\widetilde{\bm X}^{(\nu)}$, which is a collection of all the vectors at the $\nu$th iteration.
\begin{equation}
\widetilde{\bm X}^{(\nu)}=\left(\begin{array}{c}
\widetilde{\bm x}_{N_1}^{(\nu)} \\
\vdots \\
\widetilde{\bm x}_{N_2}^{(\nu)}
\end{array}
\right) \in\mathbb{R}^{(N_2-N_1+1)n}
\end{equation}
For a positive scalar $\rho$, we define a region $\mathcal{D}_{\rho}\subset\mathbb{R}\times\mathbb{R}^n$ by
\begin{equation}
\mathcal{D}_{\rho}=\{(t,{\bm x})\,|\,t\in[a,b], \|{\bm x}-{\bm x}_a\|_{\infty}\le\rho\}.
\end{equation}
Also, let
\begin{equation}
w = \max_{N_1\le i\le N_2}\sum_{j=N_1}^{N_2}|w_{ij}|.
\label{eq:wdef}
\end{equation}
We make the following assumption throughout this section.
\begin{assumption}
\label{Assumption3-1}
There exists $\rho>0$ such that the following conditions hold:
\begin{itemize}
\item[(i)] there exists a constant $M$ such that $\|{\bm f}(t,{\bm x})\|_{\infty}\le M$ holds for $\forall(t,{\bm x})\in\mathcal{D}_{\rho}$,
\item[(ii)] ${\bm f}(t,x)$ is Lipschitz continuous in $\mathcal{D}_{\rho}$, that is, there exists a constant $L>0$ such that
\begin{equation}
\|{\bm f}(t,{\bm x})-{\bm f}(t,{\bm x}')\|_{\infty}\le L\|{\bm x}-{\bm x}'\|_{\infty}, \q \forall(t,{\bm x}), (t,{\bm x}')\in \mathcal{D}_{\rho},
\label{eq:Lipschitz}
\end{equation}
\item[(iii)] the constant $M$ can be chosen so that $w\le \frac{\rho}{M}$.
\end{itemize}
\end{assumption}
Note that this is a standard assumption used to prove the existence and uniqueness of the solution of the initial value problem \eqref{eq:ODE}, except that $w$ is used instead of $b-a$ in condition (iii). As we will see at the end of subsection \ref{convergence_rate}, $w$ is roughly bounded by $b-a$ for the DE Sinc-collocation method. Thus, Assumption \ref{Assumption3-1} can be considered as a natural assumption. If the conditions (i) and (ii) hold but (iii) does not, one can consider shrinking the interval $[a,b]$ to $[a,b']$, where $a<b'<b$. Since $w$ depends linearly on $b-a$ (see the definition of $\varphi(s)$ and \eqref{eq:wij}), one can always fulfill condition (iii) by making $b'$ sufficiently close to $a$. Of course, conditions (i) and (ii) remain to hold since the region $\mathcal{D}_{\rho}$ shrinks.

We first prove the following lemma.
\begin{lemma}
\label{Lem3-2}
For $\nu=0, 1, 2, \ldots$ and $i=N_1, N_1+1, \ldots, N_2$, the vector $\tilde{\bm x}_i^{(\nu)}$ generated by the Gauss-Seidel type iteration \eqref{eq:Gauss-Seidel} satisfies $(t,\tilde{\bm x}_i^{(\nu)})\in\mathcal{D}_{\rho}$ for $\forall t\in[a,b]$.
\end{lemma}
\begin{proof}
Let $\mathcal{D}_{\rho}^{\prime}=\{{\bm x}\in\mathbb{R}^n\,|\,\|{\bm x}-{\bm x}_a\|_{\infty}\le\rho \}$. Also, let us introduce a total order $\prec$ among the index pairs $\mathcal{I}\equiv\{(\nu,i)\}_{i=N_1,\ldots,N_2}^{\nu=0,1,\ldots}$ so that $(\nu',i')\prec(\nu,i)$ when either $\nu'<\nu$ or $\nu'=\nu$ and $i'<i$ holds. Note that in the Gauss-Seidel type iteration \eqref{eq:Gauss-Seidel}, $\tilde{\bm x}_{i'}^{(\nu')} $ is used in the computation of $\tilde{\bm x}_i^{(\nu)}$ only when $(\nu',i')\in\mathcal{I}$ and $(\nu',i')\prec(\nu,i)$. We show $\tilde{\bm x}_i^{(\nu)}\in\mathcal{D}_{\rho}^{\prime}$ for $\forall(\nu,i)\in\mathcal{I}$ by induction on this total order. 

First, consider the cases of $\nu=0$ and $i\in\{N_1, N_1+1, \ldots, N_2\}$. In these cases, $\tilde{\bm x}_i^{(\nu)}\in\mathcal{D}_{\rho}^{\prime}$ holds trivially since $\tilde{\bm x}_i^{(0)}={\bm x}_a$ ($i=N_1,\ldots,N_2$).

Next, let $\nu\ge 1$ and $i\in\{N_1, N_1+1, \ldots, N_2\}$ and assume that $\tilde{\bm x}_{i'}^{(\nu')}\in\mathcal{D}_{\rho}^{\prime}$ for any $(\nu',i')\in\mathcal{I}$ with $(\nu',i')\prec(\nu,i)$. Note that this assumption holds when $\nu=0$ and $i=N_1$, which corresponds to the first vector computed in the Gauss-Seidel type iteration. By combining this assumption with Assumption \ref{Assumption3-1} and \eqref{eq:Gauss-Seidel}, we have
\begin{align}
\|\widetilde{\bm x}_i^{(\nu)}-{\bm x}_a\|_{\infty} &\le \sum_{j=N_1}^{i-1} |w_{ij}|\,\|{\bm f}(\varphi(s_j),\widetilde{\bm x}_j^{(\nu)})\|_{\infty} + \sum_{j=i}^{N_2} |w_{ij}|\,\|{\bm f}(\varphi(s_j),\widetilde{\bm x}_j^{(\nu-1)})\|_{\infty} \nonumber \\
&\le M\sum_{j=N_1}^{N_2}|w_{ij}| \nonumber \\
&\le \rho,
\end{align}
showing that $\tilde{\bm x}_i^{(\nu)} \in\mathcal{D}_{\rho}^{\prime}$.

Hence the induction is completed and we have $\tilde{\bm x}_i^{(\nu)} \in\mathcal{D}_{\rho}^{\prime}$ for $\forall(\nu,i)\in\mathcal{I}$. Thus, by definition, $(t,\tilde{\bm x}_i^{(\nu)})\in\mathcal{D}_{\rho}$ for $\forall t\in[a,b]$ and $\forall(\nu,i)\in\mathcal{I}$.
\end{proof}

In the following, we write $(w_{ij})=(\varphi'(s_j)P_j(s_i))\in\mathbb{R}^{(N_2-N_1+1)\times(N_2-N_1+1)}$ as
\begin{equation}
(w_{ij})=D+E+F,
\label{eq:DEFdefinition}
\end{equation}
where $D$, $E$ and $F$ are the diagonal part, the strictly lower triangular part and the strictly upper triangular part of $(w_{ij})$, respectively.

\subsection{Global convergence of the Gauss-Seidel type iteration}
First, we study global convergence property of the Gauss-Seidel type iteration \eqref{eq:Gauss-Seidel}. Consider an equation with $\nu$ replaced by $\nu-1$ in \eqref{eq:Gauss-Seidel} and subtract it from \eqref{eq:Gauss-Seidel} side by side. Then, we have
\begin{align}
\widetilde{\bm x}_i^{(\nu+1)}-\widetilde{\bm x}_i^{(\nu)} &= \sum_{j=N_1}^{i-1} w_{ij}\left\{{\bm f}(\varphi(s_j),\widetilde{\bm x}_j^{(\nu+1)})-{\bm f}(\varphi(s_j),\widetilde{\bm x}_j^{(\nu)})\right\} \nonumber \\
& \quad + \sum_{j=i}^{N_2} w_{ij}\left\{{\bm f}(\varphi(s_j),\widetilde{\bm x}_j^{(\nu)})-{\bm f}(\varphi(s_j),\widetilde{\bm x}_j^{(\nu-1)})\right\} \quad (i=N_1,\ldots,N_2).
\end{align}
Here, $(\varphi(s_j),\widetilde{\bm x}_j^{(\nu)})\in\mathcal{D}_{\rho}$ for $\nu=0,1,\ldots$ and $N_1\le j\le N_2$ by Lemma \ref{Lem3-2}. Thus, by taking the infinity norm of both sides and using the Lipschitz continuity \eqref{eq:Lipschitz}, we have
\begin{align}
\left\|\widetilde{\bm x}_i^{(\nu+1)}-\widetilde{\bm x}_i^{(\nu)}\right\|_{\infty} &\le \sum_{j=N_1}^{i-1} |w_{ij}|\left\|{\bm f}(\varphi(s_j),\widetilde{\bm x}_j^{(\nu+1)})-{\bm f}(\varphi(s_j),\widetilde{\bm x}_j^{(\nu)})\right\|_{\infty} \nonumber \\
& \quad + \sum_{j=i}^{N_2} |w_{ij}|\left\|{\bm f}(\varphi(s_j),\widetilde{\bm x}_j^{(\nu)})-{\bm f}(\varphi(s_j),\widetilde{\bm x}_j^{(\nu-1)})\right\|_{\infty} \nonumber \\
&\le L\sum_{j=N_1}^{i-1} |w_{ij}|\left\|\widetilde{\bm x}_j^{(\nu+1)}-\widetilde{\bm x}_j^{(\nu)}\right\|_{\infty} + L\sum_{j=i}^{N_2} |w_{ij}|\left\|\widetilde{\bm x}_j^{(\nu)}-\widetilde{\bm x}_j^{(\nu-1)}\right\|_{\infty} \nonumber \\
& \quad\quad (i=N_1,\ldots,N_2).
\label{eq:xinudiff}
\end{align}
By introducing a nonnegative vector ${\bm Z}^{(\nu)}$ by
\begin{equation}
{\bm Z}^{(\nu)}=\left(\begin{array}{c}
\left\|\widetilde{\bm x}_{N_1}^{(\nu)}-\widetilde{\bm x}_{N_1}^{(\nu-1)}\right\|_{\infty} \\
\vdots \\
\left\|\widetilde{\bm x}_{N_2}^{(\nu)}-\widetilde{\bm x}_{N_2}^{(\nu-1)}\right\|_{\infty}
\end{array}
\right) \in\mathbb{R}^{N_2-N_1+1},
\end{equation}
we can rewrite \eqref{eq:xinudiff} concisely as
\begin{equation}
{\bm Z}^{(\nu+1)}\le L|E|{\bm Z}^{(\nu+1)}+L(|D|+|F|){\bm Z}^{(\nu)}.
\label{eq:Znup1}
\end{equation}
Note that both sides of \eqref{eq:Znup1} are nonnegative. Also, since $L|E|$ is a nonngative strictly lower triangular matrix, we have
\begin{equation}
(I-L|E|)^{-1}=\sum_{k=0}^{N_2-N_1}(L|E|)^k \ge O.
\end{equation}
Thus, we can pre-multiply both sides of \eqref{eq:Znup1} by $(I-L|E|)^{-1}$ to obtain
\begin{equation}
(I-L|E|)^{-1}{\bm Z}^{(\nu+1)}\le(I-L|E|)^{-1}L|E|{\bm Z}^{(\nu+1)}+(I-L|E|)^{-1}L(|D|+|F|){\bm Z}^{(\nu)},
\end{equation}
or
\begin{equation}
{\bm Z}^{(\nu+1)}\le(I-L|E|)^{-1}L(|D|+|F|){\bm Z}^{(\nu)}.
\label{eq:Znup12}
\end{equation}
Letting
\begin{equation}
M_{\rm GS}\equiv (I-L|E|)^{-1}L(|D|+|F|)
\label{eq:MGSdefinition}
\end{equation}
and taking the infinity norm of both sides of \eqref{eq:Znup12} yields
\begin{equation}
\left\|{\bm Z}^{(\nu+1)}\right\|_{\infty} \le \left\|M_{\rm GS}\right\|_{\infty}\left\|{\bm Z}^{(\nu)}\right\|_{\infty}.
\end{equation}
Now, assume that $\|M_{\rm GS}\|_{\infty}<1$. Then,
\begin{equation}
\left\|{\bm Z}^{(\nu)}\right\|_{\infty} \le \left\|M_{\rm GS}\right\|_{\infty}^{\nu-1}\left\|{\bm Z}^{(1)}\right\|_{\infty} \rightarrow 0 \quad (\nu\rightarrow\infty).
\end{equation}
Using this, we have for $\mu>\nu\ge 0$,
\begin{align}
\left\|\widetilde{\bm X}^{(\mu)}-\widetilde{\bm X}^{(\nu)}\right\|_{\infty} &\le \sum_{k=\nu+1}^{\mu}\left\|\widetilde{\bm X}^{(k)}-\widetilde{\bm X}^{(k-1)}\right\|_{\infty} \nonumber \\
&= \sum_{k=\nu+1}^{\mu}\|{\bm Z}^{(k)}\|_{\infty} \nonumber \\
&\le \sum_{k=\nu+1}^{\mu}\|M_{\rm GS}\|_{\infty}^{k-1}\|{\bm Z}^{(1)}\|_{\infty} \nonumber \\
&= \|M_{\rm GS}\|_{\infty}^{\nu}\,\frac{1-\|M_{\rm GS}\|_{\infty}^{\mu-\nu}}{1-\|M_{\rm GS}\|_{\infty}}\,\|{\bm Z}^{(1)}\|_{\infty}
\label{eq:Cauchy}
\end{align}
Thus, $\left\|\widetilde{\bm X}^{(\mu)}-\widetilde{\bm X}^{(\nu)}\right\|_{\infty}\rightarrow 0$ when $\mu,\nu\rightarrow\infty$ and thus $\{\widetilde{\bm X}^{(\nu)}\}$ is a Cauchy sequence. Hence there exists $\widetilde{\bm X}^*\in\mathbb{R}^{(N_2-N_1+1)n}$ such that $\widetilde{\bm X}^{(\nu)}\rightarrow\widetilde{\bm X}^*$ as $\nu\rightarrow\infty$. $\widetilde{\bm X}^*$ is a fixed point of the Gauss-Seidel type iteration \eqref{eq:Gauss-Seidel} and therefore the solution to \eqref{eq:nonlinear_equation}. Furthermore, by taking the limit of $\mu\rightarrow\infty$ in \eqref{eq:Cauchy}, we have for $\nu=1, 2, \ldots$,
\begin{equation}
\left\|\widetilde{\bm X}^{(\nu)}-\widetilde{\bm X}^*\right\|_{\infty} \le \frac{\|M_{\rm GS}\|_{\infty}^{\nu}}{1-\|M_{\rm GS}\|_{\infty}}\,\|{\bm Z}^{(1)}\|_{\infty}.
\end{equation}
Summarizing the above arguments, we obtain the following theorem.
\begin{theorem}
\label{Th3-1}
Let $M_{\rm GS}$ be a matrix defined by \eqref{eq:DEFdefinition} and \eqref{eq:MGSdefinition}. If $\|M_{\rm GS}\|_{\infty}<1$, then the Gauss-Seidel type iteration \eqref{eq:Gauss-Seidel} converges and delivers the solution to \eqref{eq:nonlinear_equation}. The convergence is at least linear and an upper bound on the rate of convergence is given by $\|M_{\rm GS}\|_{\infty}$.
\end{theorem}
This theorem applies to both the DE Sinc-collocation method and the IMT-DE periodic Sinc-collocation method. Note that the matrix $(w_{ij})$, and therefore matrix $M_{\rm GS}$, differs for each method.

\subsection{Rate of convergence of the Gauss-Seidel type iteration for the DE Sinc-collocation method}
\label{convergence_rate}
Now we focus on the DE Sinc-collocation method and investigate the rate of convergence of the Gauss-Seidel type iteration \eqref{eq:Gauss-Seidel} in more detail.

We start by evaluating the infinity norm of $(I - L|E|)^{-1}$. 
To this end, 
we use the formula 
\begin{align}
(I - L |E|)^{-1} = \sum_{k=0}^{N_{2} - N_{1}} (L |E|)^{k}. 
\label{eq:formula_inverse_I_LE}
\end{align}
%To avoid overestimating the norm of the matrix in the right-hand side, 
%we estimate its entries directly instead of applying the triangle inequality. 
%To this end, 
To bound the right-hand side of this equality, 
we estimate the entries of the matrix $L|E|$. 

\begin{lemma}
\label{lem:L_e_ij_ub}
The following inequality holds:
\begin{align}
L |e_{ij}| 
\leq 
\frac{1.1L(b-a)}{2} \cdot \frac{ \frac{\pi}{2} \cosh(jh) }{ \cosh^{2}(\frac{\pi}{2} \sinh(jh)) } \cdot h
\qquad
(i > j).
\label{eq:L_e_i_j_bound}
\end{align}
\end{lemma}

\begin{proof}
Since $E$ is a part of $(w_{ij})=(\varphi'(s_j)P_j(s_i))$ with $i>j$ and $P_j(s_i)= h\left\{\frac{1}{2}+\frac{1}{\pi}{\rm Si}(\pi(i-j))\right\}>0$ when $i-j>0$, $E$ is a nonnegative matrix. 
Also, it holds that ${\rm Si}(\pi k)<{\rm Si}(\pi)<1.86$ for $k=2,3,\ldots$. 
Therefore inequality~\eqref{eq:L_e_i_j_bound} follows from 
\begin{align}
L\left| \varphi'(s_{j}) P_{j}(s_{i}) \right|
\leq
L \cdot 
\left|
\frac{b-a}{2} \cdot 
\frac{ \frac{\pi}{2} \cosh(jh) }{ \cosh^{2}(\frac{\pi}{2} \sinh(jh)) }
\right|
\cdot
h 
\cdot
\left| \frac{1}{2} + \frac{1}{\pi} \mathrm{Si}\left( \pi (i - j) \right) \right|.
\notag
\end{align}
\end{proof}

\noindent
Here we define 
\begin{align}
a_{j} 
:= 
\frac{1.1L(b-a)}{2} \cdot \frac{ \frac{\pi}{2} \cosh(jh) }{ \cosh^{2}(\frac{\pi}{2} \sinh(jh)) } \cdot h.
\label{eq:def_a_j}
\end{align}
By using $a_{j}$, 
we define a strictly lower triangular matrix $G = (g_{ij}) \in \mathbb{R}^{(N_{2} - N_{1} + 1) \times (N_{2} - N_{1} + 1)}$ by
\begin{align}
g_{ij} = 
\begin{cases}
0 & (i \leq j), \\
a_{j} & (i > j).
\end{cases} 
\label{eq:def_matrix_G}
\end{align}
Then, 
it follows from \eqref{eq:formula_inverse_I_LE} and Lemma~\ref{lem:L_e_ij_ub} that%
\footnote{
For non-negative square matrices $P = (p_{ij})$ and $Q = (q_{ij})$ with $p_{ij} \leq q_{ij} \ (\forall i,j)$, 
we have $\| P \|_{\infty} \leq \| Q \|_{\infty}$.  
This is because 
\(
\| P \|_{\infty} = \max_{i} \sum_{j} p_{ij}
\) 
and 
\(
\| Q \|_{\infty} = \max_{i} \sum_{j} q_{ij}
\).
}
\begin{align}
\left\| (I - L |E|)^{-1} \right\|_{\infty} 
\leq 
\left\| 
\sum_{k=0}^{N_{2} - N_{1}} G^{k}
\right\|_{\infty}.
\label{eq:key_est_inv_I_LE}
\end{align}
To bound the right-hand side of this inequality, 
we use the following lemma, 
which is proved in Section~\ref{sec:proof_norm_series_of_G}. 

\begin{lemma}
\label{thm:norm_series_of_G}
Let $G = (g_{ij})$ be the matrix defined by~\eqref{eq:def_matrix_G}. 
Then the following equality holds:
\begin{align}
\left\| 
\sum_{k=0}^{N_{2} - N_{1}} G^{k}
\right\|_{\infty}
=
\prod_{r = N_{1}}^{N_{2}-1} (1 + a_{r}).
\notag
\end{align}
\end{lemma}

\noindent
Then, 
it follows from \eqref{eq:key_est_inv_I_LE} and Lemma~\ref{thm:norm_series_of_G} that
\[
\| (I - L |E|)^{-1} \|_{\infty}
\leq 
\prod_{r = N_{1}}^{N_{2}-1} (1 + a_{r})
\leq 
\prod_{r = N_{1}}^{N_{2}-1} \exp(a_{r})
=
\exp \left(
\sum_{r = N_{1}}^{N_{2}-1} a_{r}
\right)
\leq 
\exp \left(
\sum_{r = N_{1}}^{N_{2}} a_{r}
\right). 
\]
The second inequality is owing to the general inequality $1+x \leq \exp(x)$ for any $x \in \mathbb{R}$. 
By the definition of $a_{j}$ in \eqref{eq:def_a_j}, 
we have
\begin{align}
\sum_{r = N_{1}}^{N_{2}} a_{r}
& =
\frac{1.1L(b-a)}{2} \cdot h
\sum_{r = N_{1}}^{N_{2}} \frac{ \frac{\pi}{2} \cosh(rh) }{ \cosh^{2}(\frac{\pi}{2} \sinh(rh)) }
\notag \\
& \leq
\frac{1.1L(b-a)}{2} \cdot h
\sum_{r = -\infty}^{\infty} \frac{ \frac{\pi}{2} \cosh(rh) }{ \cosh^{2}(\frac{\pi}{2} \sinh(rh)) }
\notag \\
& \leq
\frac{1.1L(b-a)}{2} 
\left(
h \cdot \frac{ \frac{\pi}{2} \cosh(0) }{ \cosh^{2}(\frac{\pi}{2} \sinh(0)) }
+ 
\int_{-\infty}^{\infty} \frac{\frac{\pi}{2}\cosh(s)}{\cosh^2\left(\frac{\pi}{2}\sinh(s)\right)}\,{\rm d}s
\right)
\nonumber \\
& = 
\frac{1.1L(b-a)}{2} \left( \frac{\pi}{2} h + \int_{-1}^1{\rm d}t \right) \nonumber \\
& = 
1.1L(b-a) \left( \frac{\pi}{4} h  + 1 \right)
\notag \\
& \leq
1.1L(b-a) ( h  + 1 ), 
\label{eq:bound_sum_a_j}
\end{align}
where the second inequality is owing to 
Lemma~\ref{Lem_diff_DE_monotone}. 
By combining the above arguments, we obtain
\begin{align}
\| (I - L |E|)^{-1} \|_{\infty}
\leq 
\exp \left(
1.1L(b-a) ( h  + 1 )
\right).
\label{eq:ImEinvNorm}
\end{align}

Next, we evaluate the infinity norm of $|D|+|F|$. First, from Lemma \ref{LemA-1} and the fact that ${\rm Si}(x)$ is an odd function, we have for $k=1,2, \ldots$,
\begin{equation}
\left|\frac{\pi}{2}+{\rm Si}(-k\pi)\right|<\frac{1}{k\pi}.
\label{eq:Sideviation}
\end{equation}
Using this, the sum of the elements of in the $i$th row of $|D|+|F|$ can be evaluated as follows.
\begin{align}
\sum_{j=N_{1}}^{N_{2}} (|d_{ij}|+|f_{ij}|) &= \sum_{j=i}^{N_{2}} |w_{ij}| \nonumber \\
&= \frac{(b-a)}{2}\sum_{j=i}^{N_{2}} \frac{\frac{\pi}{2}\cosh(jh)}{\cosh^2\left(\frac{\pi}{2}\sinh(jh)\right)}\cdot h\left|\frac{1}{2}+\frac{1}{\pi}{\rm Si}(\pi(i-j))\right| \nonumber \\
&\le \frac{(b-a)}{2}\cdot\frac{\pi}{2}h\sum_{j=i}^{N_{2}} \left|\frac{1}{2}+\frac{1}{\pi}{\rm Si}(\pi(i-j))\right| \nonumber \\
&< \frac{(b-a)}{2}\cdot\frac{\pi}{2}h\left\{\frac{1}{2}+\frac{1}{\pi}\sum_{j=i+1}^{N_{2}} \frac{1}{(j-i)\pi}\right\} \nonumber \\
&< \frac{(b-a)}{2}\cdot\frac{\pi}{2}h\left\{\frac{1}{2}+\frac{1}{\pi^2}(1+\log(N_{2}-i))\right\} \nonumber \\
&\le (b-a)h\left\{\frac{\pi}{8}+\frac{1}{4\pi}(1+\log(N_{2} - N_{1}))\right\} \q (i=N_{2},\ldots,N_{1}),
\label{eq:dpfbound2}
\end{align}
where in the first inequality, we used $\forall s\in\mathbb{R},\, \cosh(s)/\cosh^2\left(\frac{\pi}{2}\sinh(s)\right) \le 1$, which can be shown as
\begin{equation}
\cosh^2\left(\frac{\pi}{2}\sinh(s)\right) \ge \cosh\left(\frac{\pi}{2}\sinh(s)\right) \ge \cosh\left(\frac{\pi}{2}s\right) \ge \cosh(s).
\end{equation}
From \eqref{eq:dpfbound2}, we have
\begin{equation}
\|D+F\|_{\infty}=\max_{N_{1} \le i\le N_{2}}\sum_{j=N_{1}}^{N_{2}}(|d_{ij}|+|f_{ij}|) \le (b-a)h\left\{\frac{\pi}{8}+\frac{1}{4\pi}(1+\log(N_{2} - N_{1}))\right\}.
\label{eq:DpFNorm2}
\end{equation}

By taking the infinity norm of \eqref{eq:MGSdefinition} and inserting \eqref{eq:ImEinvNorm} and \eqref{eq:DpFNorm2}, we obtain the following theorem on the upper bound on the infinity norm of $M_{\rm GS}$.
\begin{theorem}
\label{Th3-2}
Suppose that Assumption \ref{Assumption3-1} 
%and $1.1L(b-a)<1$ 
hold. Then, for the DE Sinc-collocation method, the infinity norm of the matrix $M_{\rm GS}$ defined by \eqref{eq:MGSdefinition} is bounded as follows.
\begin{equation}
\|M_{\rm GS}\|_{\infty} 
\le 
\exp \left(
1.1L(b-a) ( h  + 1 )
\right)
L(b-a)h \left\{\frac{\pi}{8}+\frac{1}{4\pi}(1+\log(N_{2} - N_{1}))\right\}.
\label{eq:MGSNorm2}
\end{equation}
\end{theorem}

It is to be noted that this bound does not depend on $n$, the dimension of ${\bm x}(t)$.
%This is in contrast to the similar bound for the Jacobi-type iteration (Picard iteration) given in \cite[Eq.~(4.11)]{Ogata23}.
When $L(b-a)=\frac{1}{2}$, 
$N = -N_{1} = N_{2} = 64$ and 
$h$ is set to $\log N/N$, for example, we have $\log(2N)\simeq 4.852$ and $h\simeq 0.06498$ and therefore
\begin{equation}
\|M_{\rm GS}\|_{\infty} 
\le 
\exp \left( 1.1\cdot \frac{1}{2} \cdot(0.06498+1) \right) \cdot
\frac{1}{2}\cdot 0.06498 \cdot
\left(\frac{\pi}{8}+\frac{1}{4\pi}(1+4.852))\right) 
\simeq 0.05010,
\label{eq:MGSboundsample}
\end{equation}
which means that the error is reduced to about $1/20$ per iteration. 
This is a very fast convergence rate compared with the usual convergence rate of the Gauss-Seidel method, which is typically $>0.99$ when solving the Poisson equation using the finite difference method (see, for example, \cite[Table 3.1]{Sugihara09}). Note also that $\|M_{\rm GS}\|_{\infty}$ {\it decreases} as $N$ increases, since when $h=O(\log N/N)$, the upper bound in \eqref{eq:MGSNorm2} is $O((\log N)^2/N)$. This is also in clear contrast to the Poisson case, for which $\|M_{\rm GS}\|_{\infty}$ {\it increases} with the matrix size.

Finally, we give a rough upper bound on $w$ defined by \eqref{eq:wdef}. 
To this end, we prove the following lemma, 
which is a straightforward consequence of Lemma~\ref{lem:L_e_ij_ub}. 

\begin{lemma}
\label{thm:Ebound}
The infinity norm of the matrix $E$ is bounded as follows: 
\[
\| E \|_{\infty} \leq 1.1(b-a) ( h  + 1 ). 
\]
\end{lemma}

\begin{proof}
From Lemma~\ref{lem:L_e_ij_ub} and~\eqref{eq:def_a_j}, we have
\begin{align}
\| E \|_{\infty} 
= 
\max_{N_{1} \leq i \leq N_{2}} \sum_{j = N_{1}}^{N_{2}} |e_{ij}| 
\leq
\max_{N_{1} \leq i \leq N_{2}} \sum_{j = N_{1}}^{N_{2}} \frac{a_{j}}{L} 
\leq 
1.1(b-a) ( h  + 1 )
\notag
\end{align}
in the same manner as \eqref{eq:bound_sum_a_j}. 
\end{proof}

\noindent
Using Lemma~\ref{thm:Ebound} and \eqref{eq:DpFNorm2} and noting that typically $h\ll 1$ and $h\log N\ll 1$, we have
\begin{equation}
w=\|(w_{ij})\|_{\infty} \le \|E\|_{\infty}+\|D+F\|_{\infty} \simeq \|E\|_{\infty} \le 1.1(b-a) ( h  + 1 ).
\end{equation}
Thus, $w$ is roughly bounded by $b-a$, as we insisted after Assumption \ref{Assumption3-1}.

\subsection{Intuitive explanation for the efficiency}
Lemma~\ref{thm:Ebound} and 
Eq.~\eqref{eq:DpFNorm2} suggest that the matrix $(w)_{ij}=D+E+F$ is almost lower triangular and this is the reason for the high efficiency of the Gauss-Seidel type iteration. We will explain intuitively why this nearly lower triangular structure arises.

Let us consider a simplified fictitious method where the DE transformation \eqref{eq:DE} is replaced by the identity transformation $\varphi(s)=s$ and $p_j(s)$ is a piecewise linear function with $p_j(s_i)=\delta_{ij}$. Then, it is easy to see that $P_j(s_i)=0$ for $i<j$ and the matrix $(w_{ij})$ is lower triangular. In fact, this method is nothing but the implicit Euler method and the lower triangularity means that the method respects causality, in the sense that it does not use the function values at the $j$th time step, where $j>i$, to compute the function values at the $i$th time step.

The DE Sinc-collocation method achieves very high accuracy by breaking this type of causality to a certain extent. But the weight of function values at future time steps ($P_j(s_i)$ for $i<j$) must be small because it is unnatural to use such information from a physical point of view. This explains the nearly lower triangular structure of $(w_{ij})$.

\section{Proof of Lemma~\ref{thm:norm_series_of_G}}
\label{sec:proof_norm_series_of_G}

The following lemma is a key to prove Lemma~\ref{thm:norm_series_of_G}. 

\begin{lemma}
\label{thm:calc_series_of_G}
Let $G = (g_{ij})$ be the matrix defined by~\eqref{eq:def_matrix_G}. 
In the case that $i>j$, we have
\begin{align}
\left(  \sum_{\ell=1}^{N_{2} - N_{1}} G^{\ell} \right)_{ij}
= a_j\prod_{r=j+1}^{i-1}(1+a_r).
\label{eq:calc_series_of_G}
\end{align}
When $i = j+1$, the right-hand side is regarded as $a_{j}$ by convention. 
In the case that $i\le j$, the entry is zero. 
\end{lemma}

\begin{proof}
Let $K = N_{2} - N_{1} + 1$. 
Since $G$ is a strictly lower triangular matrix, $G^{\ell}=O$ for $\ell\ge K$. Therefore, $(I-G)^{-1}$ can be expanded into a finite Neumann series, and the left-hand side of Eq.~\eqref{eq:calc_series_of_G} can be written as $(I-G)^{-1}-I$. Since subtracting $I$ does not change the strictly lower triangular elements, it suffices to find the explicit formula for the strictly lower triangular elements of $(I-G)^{-1}$. Note that from the strict lower triangularity of $G$, the left-hand side of Eq.~\eqref{eq:calc_series_of_G} is trivially 0 when $i\le j$.

Let $J\in\mathbb{R}^{K\times K}$ be a matrix whose lower bidiagonal elements are 1 ($(J)_{i+1,i}=1$) and all other elements are 0. This matrix has the following properties:
\begin{itemize}
\item $J^{\ell}$ ($1\le\ell\le K-1$) is a matrix with elements 1 only on the $\ell$-th lower subdiagonal ($(J^{\ell})_{j+\ell,j}=1$).
\item Let $D\in\mathbb{R}^{K\times K}$ be a diagonal matrix, and let $D^<$ be the matrix obtained by shifting all diagonal elements of $D$ one position up and to the left 
($(D^<)_{jj}=(D)_{j+1,j+1}$ for $1\le j\le K-1$, and $(D^<)_{KK}=0$). 
Similarly, 
let $D^{<^{k}}$ be the matrix obtained by repeated application of the shift $k$ times 
($(D^{<^{k}})_{jj}=(D)_{j+k,j+k}$ for $1\le j\le K-k$, and $(D^{<^{k}})_{jj}=0$ for $K-k+1 \le j\le K$). 
Then, $DJ=JD^<$ and $D^{<^{k}} J = J D^{<^{k+1}}$ hold.
\end{itemize}

We set $D={\rm diag}(a_{N_{1}},\ldots,a_{N_{2}-1},0)$. 
Since the $\ell$-th lower subdiagonal of $G$ is given by $J^{\ell}D$, 
we have $G=\sum_{\ell=1}^{K-1} J^{\ell}D$. 
Then we can transform $I-G$ as follows:
\begin{align}
I-G &= I-\sum_{\ell=1}^{K-1} J^{\ell}D \nonumber \\
&= I+D - \left(\sum_{\ell=0}^{K-1} J^{\ell}\right)D \nonumber \\
&= I+D-(I-J)^{-1}D \nonumber \\
&= I+D-(I-J)^{-1}( (I-J)D + JD ) \nonumber \\
&= I-(I-J)^{-1}JD \nonumber \\
&= (I-J)^{-1} ((I-J) - JD) \nonumber \\
&= (I-J)^{-1}(I-J(I+D)).
\end{align}
Consequently, we have
\begin{align}
(I-G)^{-1}&=(I-J(I+D))^{-1}(I-J) \nonumber \\
&= \left(\sum_{\ell=0}^{K-1}(J(I+D))^{\ell}\right)(I-J) \nonumber \\
&= \sum_{\ell=0}^{K-1}\underbrace{J(I+D)J(I+D)\cdots J(I+D)}_{\text{$\ell$ times}}-\sum_{\ell=0}^{K-1}\underbrace{J(I+D)J(I+D)\cdots J(I+D)}_{\text{$\ell$ times}}J \nonumber \\
&= \sum_{\ell=0}^{K-1}J^{\ell}(I+D^{<^{\ell-1}})\cdots(I+D^<)(I+D)-\sum_{\ell=0}^{K-1}J^{\ell+1}(I+D^{<^{\ell}})\cdots(I+D^{<^2})(I+D^<) \nonumber \\
&= I+\sum_{\ell=1}^{K-1}J^{\ell}(I+D^{<^{\ell-1}})\cdots(I+D^<)D.
\end{align}
In the fourth equality, we used $DJ=JD^<$ and $D^{<^{k}} J = J D^{<^{k+1}}$. 
If we define $a_j=0$ for $j\ge K$, then $(I+D^{<^{\ell-1}})\cdots(I+D^<)D$ is a diagonal matrix whose $(j,j)$ element is $a_j\prod_{r=j+1}^{j+\ell-1}(1+a_r)$. Therefore the $(j+\ell,j)$ element of $J^{\ell}(I+D^{<^{\ell-1}})\cdots(I+D^<)D$, which is the $(j+\ell,j)$ element of $(I-G)^{-1}$, is $a_j\prod_{r=j+1}^{j+\ell-1}(1+a_r)$. By setting $i>j$ and $\ell=i-j$, we obtain the $(i,j)$ element of $(I-G)^{-1}$, which is $a_j\prod_{r=j+1}^{i-1}(1+a_r)$. This completes the proof of Lemma \ref{thm:calc_series_of_G}.
\end{proof}

By using Lemma~\ref{thm:calc_series_of_G}, 
we prove Lemma~\ref{thm:norm_series_of_G} as follows.

\begin{proof}[Proof of Lemma~\ref{thm:norm_series_of_G}]
Because
\begin{align}
\left\| 
\sum_{k=0}^{N_{2} - N_{1}} G^{k}
\right\|_{\infty}
= 
\max_{N_{1} \leq i \leq N_{2}} \sum_{j=N_{1}}^{N_{2}} 
\left( \sum_{k=0}^{N_{2} - N_{1}} G^{k} \right)_{ij},
\label{eq:norm_series_of_G_expr}
\end{align}
we calculate the sum in the right-hand side. 
By Lemma \ref{thm:calc_series_of_G}, 
we have
\begin{align}
\sum_{j=N_{1}}^{N_{2}} 
\left( \sum_{k=0}^{N_{2} - N_{1}} G^{k} \right)_{ij}
& = 
\sum_{j=N_{1}}^{N_{2}} 
\left( I + \sum_{\ell=1}^{N_{2} - N_{1}} G^{\ell} \right)_{ij}
= 
\sum_{j=N_{1}}^{i} 
\left( I + \sum_{\ell=1}^{N_{2} - N_{1}} G^{\ell} \right)_{ij}
\notag \\
& =
1 + 
\sum_{j=N_{1}}^{i-1} 
\left( \sum_{\ell=1}^{N_{2} - N_{1}} G^{\ell} \right)_{ij}
= 
1 + 
\sum_{j=N_{1}}^{i-1} 
a_j\prod_{r=j+1}^{i-1}(1+a_r).
\label{eq:sum_series_of_G_expr}
\end{align}

To calculate this expression, 
we define
\begin{align}
R_{i} 
:= 
\sum_{j=N_{1}}^{i-1} 
a_j\prod_{r=j+1}^{i-1}(1+a_r)
\label{eq:def_R_i}
\end{align}
and introduce the partial products
\[
P_{N_{1}-1}:=1, 
\qquad
P_k:=\prod_{r=N_{1}}^{k}(1+a_r)
\qquad
(k=N_{1}, \ldots, N_{2}-1). 
\]
Then we have
\[
R_i
=
\sum_{j=N_{1}}^{i-1} a_{j} \frac{P_{i-1}}{P_j}
=
P_{i-1}\sum_{j=N_{1}}^{i-1}\frac{a_j}{P_j}.
\]
Furthermore, 
using $P_j=P_{j-1}(1+a_j)$, we have
\[
\frac{1}{P_{j-1}}-\frac{1}{P_j}
=\frac{(1+a_j)}{P_{j}}-\frac{1}{P_{j}}
=\frac{a_j}{P_j}.
\]
By combining these, we can deduce a simple expression of $R_{i}$ as follows:
\begin{align}
R_{i}
=
P_{i-1}
\sum_{j=N_{1}}^{i-1}\frac{a_j}{P_j}
=
P_{i-1}
\sum_{j=N_{1}}^{i-1}\left( \frac{1}{P_{j-1}}-\frac{1}{P_j} \right)
=
P_{i-1}
\left( 
\frac{1}{P_{N_{1}-1}}-\frac{1}{P_{i-1}}
\right)
=P_{i-1} - 1.
\label{eq:R_i_simple_expr}
\end{align}

From 
\eqref{eq:norm_series_of_G_expr}, 
\eqref{eq:sum_series_of_G_expr}, 
\eqref{eq:def_R_i}, and
\eqref{eq:R_i_simple_expr}, 
we have
\begin{align}
\left\| 
\sum_{k=0}^{N_{2} - N_{1}} G^{k}
\right\|_{\infty}
=
\max_{N_{1} \leq i \leq N_{2}} P_{i-1}
=
\max_{N_{1} \leq i \leq N_{2}} \prod_{r=N_{1}}^{i-1}(1+a_r)
=
\prod_{r=N_{1}}^{N_{2}-1}(1+a_r).
\notag
\end{align}
Hence we obtain the conclusion. 
\end{proof}

\section{Numerical examples}
\label{Results}
In this section, we present some numerical examples that confirm the validity of Theorems \ref{Th3-1} and \ref{Th3-2}. The program was written in FORTARN using double precision arithmetic, compiled with Intel Fortran Ver.~11.1 and executed on a workstation with the Xeon E5-2660 v2 CPU. The values of the ${\rm Si}(x)$ function were computed by WolframAlpha \cite{WolframAlpha}.

In the following, 
we assume that 
$N_{1} = -N$ and $N_{2} = N$ for a positive integer $N$.
Furthermore, 
${\bm x}(t)$ (or $x(t)$ in the scalar case) denotes the exact solution of the given initial value problem and $\widetilde{\bm x}_i^*$ (or $\widetilde{x}_i^*$) denotes the approximate solution by the DE Sinc-collocation method at the $i$th sample point $t_i=\varphi(ih)$ ($-N\le i\le N$), where $\varphi$ is defined by \eqref{eq:DE}. 
Note that $\widetilde{\bm x}_i^*$ is defined as the exact solution of the nonlinear simultaneous equations \eqref{eq:nonlinear_equation}. In the numerical experiments, $\widetilde{\bm x}_i^*$ was computed approximately by performing the Gauss-Seidel iteration 10 times. We evaluate two kinds of errors. The first one is the error of the numerical solution defined by
\begin{align}
E_1(N)=\max_{-N\le i\le N}\left\|\widetilde{\bm x}_i^*-{\bm x}(t_i)\right\|_{\infty},
\end{align}
The second one is the error of the Gauss-Seidel type iteration defined by
\begin{align}
E_2(\nu)=\left\|\widetilde{\bm x}_i^{(\nu)}-\widetilde{\bm x}_i^*\right\|_{\infty},
\end{align}
where $\widetilde{\bm x}_i^{(\nu)}$ is the $\nu$th iterate. We also investigate the actual value of $\|M_{\rm GS}\|_{\infty}$ as a function of $N$ and compare it with the bound \eqref{eq:MGSNorm2}. In all of the examples, the step size was chosen as $h=\log N/N$.

\paragraph{Example 1}
The first example is a simple linear ordinary differential equation with a scalar variable:
\begin{align}
\frac{{\rm d}x}{{\rm d}t}=x \quad (0\le t\le 1/2), \quad x(0)=1,
\label{eq:example1}
\end{align}
with the exact solution $x(t)=e^t$. In this case, Assumption \ref{Assumption3-1} is satisfied with $\rho=1$, $M=2$, $L=1$ and $w\simeq b-a=\frac{1}{2}\le \rho/M$. Also, it holds that $1.1L(b-a)=\frac{1.1}{2}<1$. Thus, we can apply Theorem \ref{Th3-2} with $h=\log N/N$ and obtain the bound $\|M_{\rm GS}\|_{\infty}<1$ for $N\ge 2$. Hence, convergence of the Gauss-Seidel type iteration is guaranteed by Theorem \ref{Th3-1}.

In Fig.~\ref{fig:Ex1-1}, we show the error $E_1(N)$ as a function of $N$. 
%The error decreases almost exponentially with $N$, in accordance with the convergence theory of the DE Sinc-collocation method \cite[Theorem 10]{Okayama18}.
For $N\ge 32$, the error decreases nearly to the roundoff error level, illustrating the power of the DE Sinc-collocation method. Fig.~\ref{fig:Ex1-2} shows the values of $\|M_{\rm GS}\|_{\infty}$ along with the bound \eqref{eq:MGSNorm2} as functions of $N$. The bound predicts the behavior of $\|M_{\rm GS}\|_{\infty}$ quite well, although it is about three times larger than the actual value. Convergence of the Gauss-Seidel type iteration is shown in Fig.~\ref{fig:Ex1-3} for the case of $N=64$. The convergence is linear, as expected from Theorem \ref{Th3-1}, and the error $E_2(\nu)$ gets reduced by almost two orders of magnitude after each iteration. This is roughly in agreement with the fact that $\|M_{\rm GS}\|_{\infty}\simeq 0.02$ in this case.

\begin{figure}[h]
\centering
\begin{minipage}[b]{0.47\columnwidth}
    \centering
    \includegraphics[width=0.9\columnwidth]{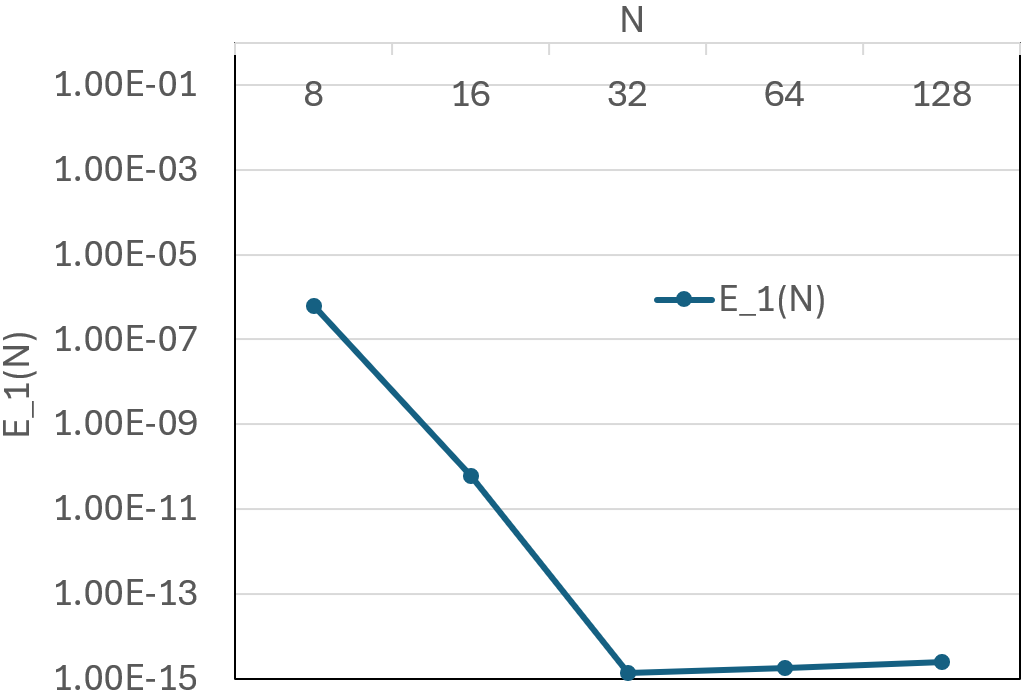}
    \caption{Error of the numerical solution as a function of $N$ for Example 1.}
    \label{fig:Ex1-1}
\end{minipage}
\hspace{3mm}
\begin{minipage}[b]{0.47\columnwidth}
    \centering
    \includegraphics[width=0.9\columnwidth]{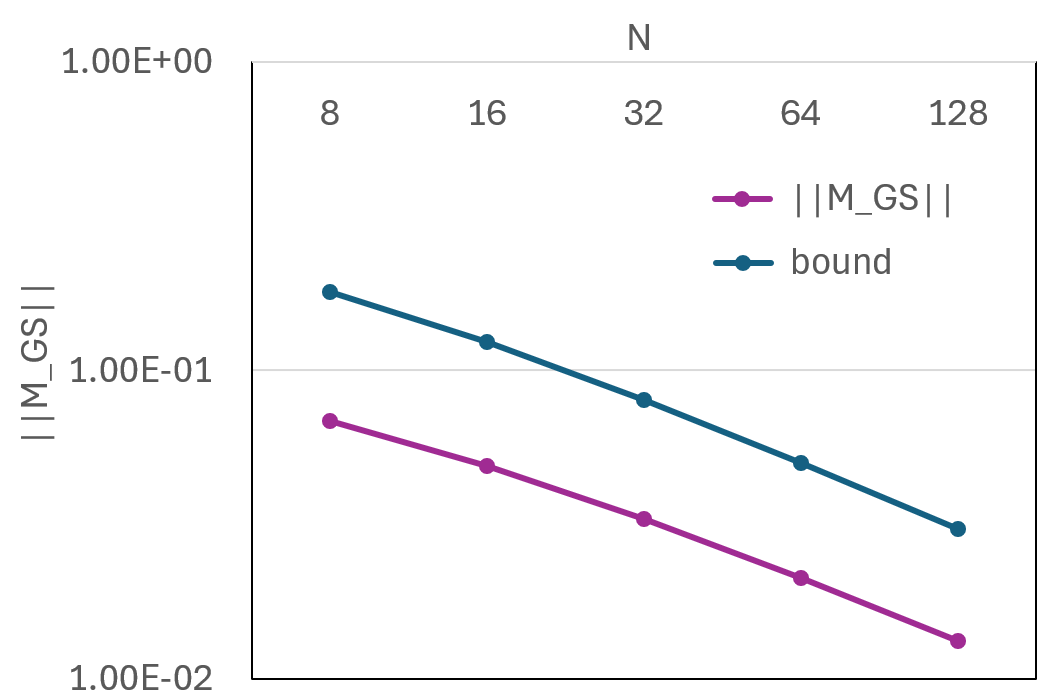}
    \caption{$\|M_{\rm GS}\|_{\infty}$ and its upper bound as a function of $N$ for Example 1.}
    \label{fig:Ex1-2}
\end{minipage}
\end{figure}

\begin{figure}[h]
\centering
\includegraphics[width=0.42\linewidth]{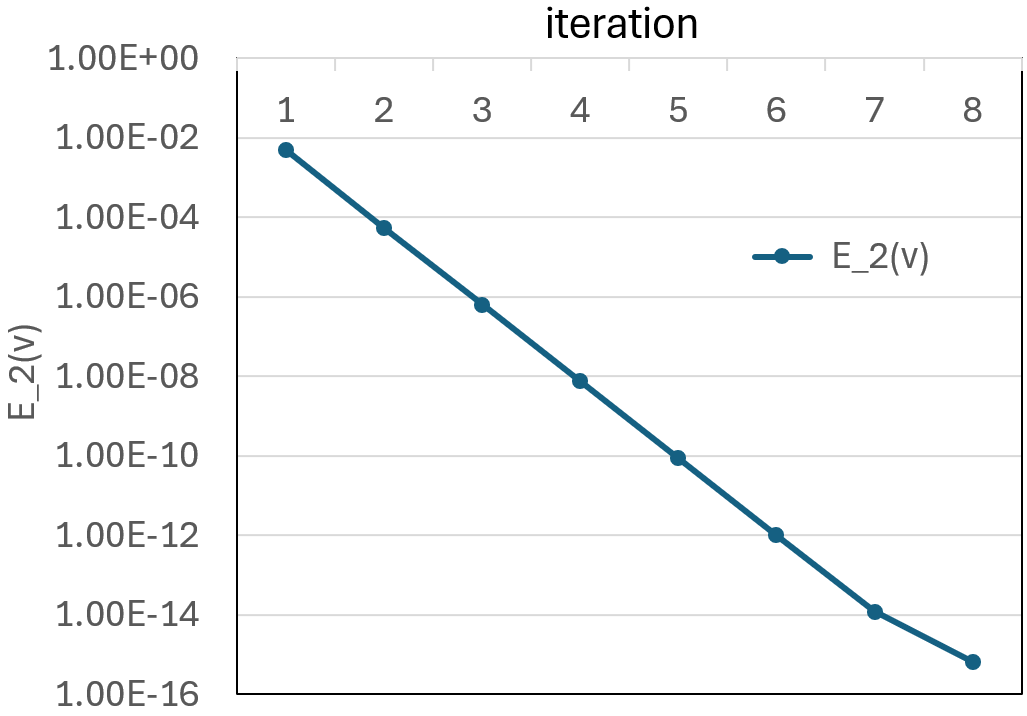}
\caption{Convergence of the Gauss-Seidel type iteration for Example 1.}
\label{fig:Ex1-3}
\end{figure}

\paragraph{Example 2}
The second example is the following initial value problem arising from spatial discretization of the heat equation using $n$ points. We assume that $n$ is odd.
\begin{align}
& \frac{{\rm d}{\bm x}}{{\rm d}t}=A{\bm x} \quad (0\le t\le 1/8), \quad {\bm x}(0)={\bm x}_0 \in\mathbb{R}^n, \nonumber \\
& A=\left(
\begin{array}{cccc}
-2 & 1 & & \\
1 & -2 & \ddots & \\
& \ddots & \ddots & 1 \\
& & 1 & -2
\end{array}
\right) \in\mathbb{R}^{n\times n}, \quad
({\bm x}_0)_k= \begin{cases}
1 & (k=(n+1)/2), \\
0 & \mbox{(otherwise)},
\end{cases}
\label{eq:Ex2-1}
\end{align}
where $({\bm x}_0)_i$ is the $i$th element of ${\bm x}_0$. Since $\|A\|_{\infty}=4$, we can choose $L=4$ and Assumption \ref{Assumption3-1} is satisfied with $\rho=1$, $M=6$\footnote{Consider the case of $({\bm x})_{(n+1)/2}=2$ and $({\bm x})_{(n+1)/2-1}=({\bm x})_{(n+1)/2+1}=-1$.} and $w\simeq b-a=\frac{1}{8}\le \rho/M$. Also, the condition $1.1L(b-a)=4\cdot\frac{1.1}{8}<1$ is met. The value of $L(b-a)$ is the same as that in Example 1, so we have $\|M_{\rm GS}\|_{\infty}<1$, which fulfills the condition of Theorem \ref{Th3-1}. Noting that the eigenvalues of $A$ are $\lambda_{\ell}=-4\sin^2\frac{\ell\pi}{2(n+1)}$ ($\ell=1,\ldots,n$) and the corresponding eigenvectors are $({\bm v}_{\ell})_k=\sqrt{\frac{2}{n+1}}\sin\frac{k\ell\pi}{n+1}$, we can write the solution of \eqref{eq:Ex2-1} as
\begin{equation}
({\bm x}(t))_k=\frac{2}{n+1}\sum_{\ell=1}^n\sin\frac{k\ell\pi}{n+1}\sin\frac{\ell\pi}{2}\exp\left(-4t\sin^2\frac{\ell\pi}{2(n+1)}\right).
\end{equation}

For this problem, $E_1(N)$ and $E_2(\nu)$ are shown in Figs.~\ref{fig:Ex2-1} and \ref{fig:Ex2-3}, respectively. Note that the matrix $M_{\rm GS}$ for this problem is exactly the same as that for Example 1 because the matrix depends only on $N$ and $L(b-a)$. Hence, the figure corresponding to Fig.~\ref{fig:Ex1-2} is not shown. In the figures, both the results of $n=11$ and $n=101$ are shown. The behaviors of the errors look similar to those in Example 1 and seem to be independent of $n$, as suggested by Theorem \ref{Th3-2}. Note that the graphs for $n=11$ and $n=101$ are completely overlapped in Fig.~\ref{fig:Ex2-3}. It is remarkable that the error $E_1(N)$ decreases to the roundoff error level even for the case of $n=101$.

\begin{figure}[h]
\centering
\begin{minipage}[b]{0.46\columnwidth}
    \centering
    \includegraphics[width=0.9\columnwidth]{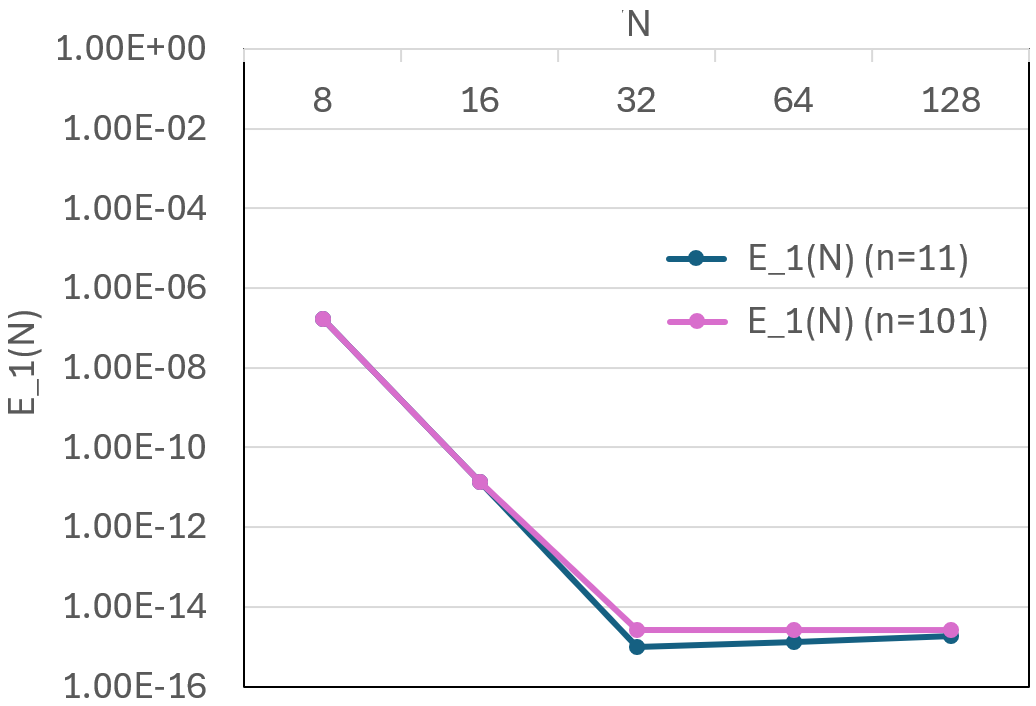}
    \caption{Error of the numerical solution as a function of $N$ for Example 2.}
    \label{fig:Ex2-1}
\end{minipage}
\hspace{3mm}
\begin{minipage}[b]{0.46\columnwidth}
    \centering
    \includegraphics[width=0.9\columnwidth]{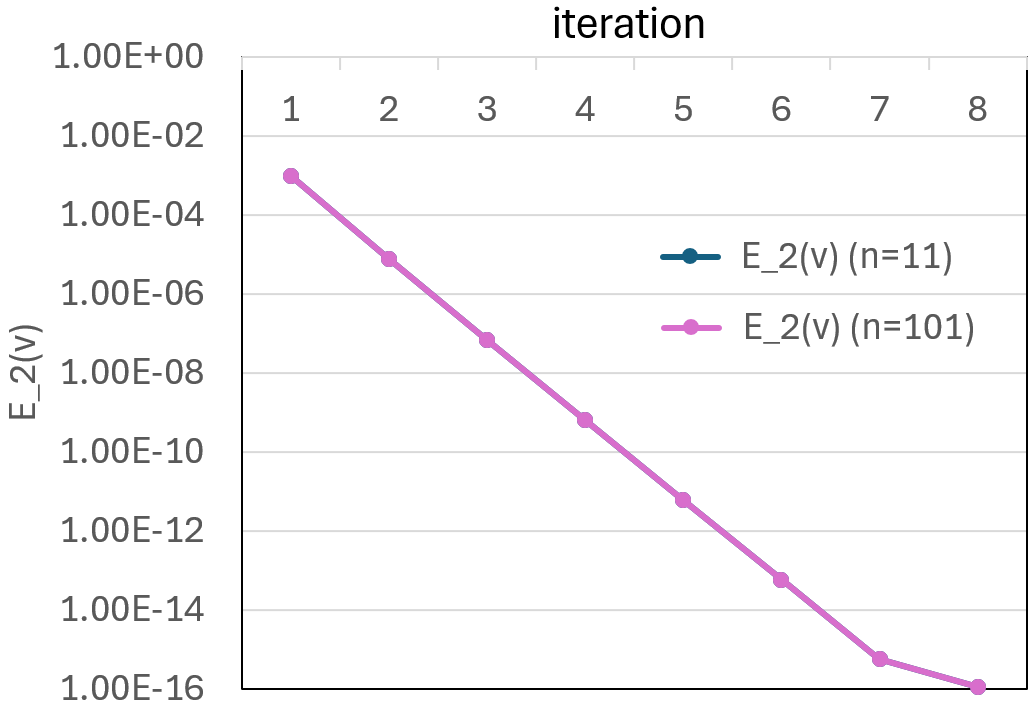}
    \caption{Convergence of the Gauss-Seidel type iteration for Example 2.}
    \label{fig:Ex2-3}
\end{minipage}
\end{figure}

\paragraph{Example 3}
The final example is the Lotka-Volterra predator-prey model with three species.
\begin{align}
\begin{dcases}
& \frac{{\rm d}x_1}{{\rm d}t}=x_1(t)x_2(t), \\
& \frac{{\rm d}x_2}{{\rm d}t}=x_2(t)( x_3(t)-x_1(t)), \\
& \frac{{\rm d}x_3}{{\rm d}t}=-x_3(t)x_2(t) \q (0\le t\le 1), \\
& x_1(0)=2, \q x_2(0)=\frac{1}{2}, \q x_3(0)=\frac{3}{2}.
\end{dcases}
\label{eq:LV3}
\end{align}
This is a system of nonlinear equations that has an analytical solution given by \eqref{eq:LV3sol1}, \eqref{eq:LV3sol2} and \eqref{eq:LV3sol3}. The derivation of this solution is explained in Appendix \ref{derivation}. In this case, if we choose $\rho=1$, then $M=(2+\rho)(\frac{1}{2}+\rho)=\frac{9}{2}$ and $L=\frac{3}{2}+4\rho=\frac{11}{2}$ from Appendix \ref{derivation}. Hence, we can set the integration interval $[a,b]$ to $[0,\rho/M]=[0,\frac{2}{9}]$. Then, $(b-a)L=\frac{11}{9}$ and we have $\|M_{\rm GS}\|_{\infty}<1$ for $N\ge 16$. In the numerical experiment, we also checked the case of $N=8$.

%${\bm x}\in\mathcal{D}_{\rho}^{\prime}$ implies $x_1\in[3/2,5/2]$, $x_2\in[0,1]$ and $x_3\in[1,2]$, which permits us to choose $M=5/2$ and $L=3/2$ and hence $b-a=\rho/M=1/5$. Thus, we set the integration interval to $[0,1/5]$. Then, $1.1L(b-a)=3.3/10<1$ and thus $\|M_{\rm GS}\|_{\infty}<1$ for $N\ge 2$.

%if we choose $\rho=1/2$, ${\bm x}\in\mathcal{D}_{\rho}^{\prime}$ implies $x_1\in[3/2,5/2]$, $x_2\in[0,1]$ and $x_3\in[1,2]$, which permits us to choose $M=5/2$ and $L=3/2$ and hence $b-a=\rho/M=1/5$. Thus, we set the integration interval to $[0,1/5]$. Then, $1.1L(b-a)=3.3/10<1$ and thus $\|M_{\rm GS}\|_{\infty}<1$ for $N\ge 2$.

%For a nonlinear system like this, it is in general difficult to find a proper value of the Lipshitz constant $L$. In this case, however, it is easy to see that $0<x_2(t)\le\frac{1}{2}$ and $|x_3(t)-x_1(t)|<\frac{5}{2}$. Accordingly, by assuming that the vector $()$

In Figs.~\ref{fig:Ex3-1}, \ref{fig:Ex3-2} and \ref{fig:Ex3-3}, we show the graphs of $E_1(N)$, $\|M_{\rm GS}\|_{\infty}$ as a function of $N$, and $E_2(\nu)$ for the case of $N=64$, respectively. While this is a nonlinear problem, these graphs show the same tendency as those of Examples 1 and 2, illustrating the validity of the analysis given in the previous section. 

\begin{figure}[h]
\centering
\begin{minipage}[b]{0.47\columnwidth}
    \centering
    \includegraphics[width=0.9\columnwidth]{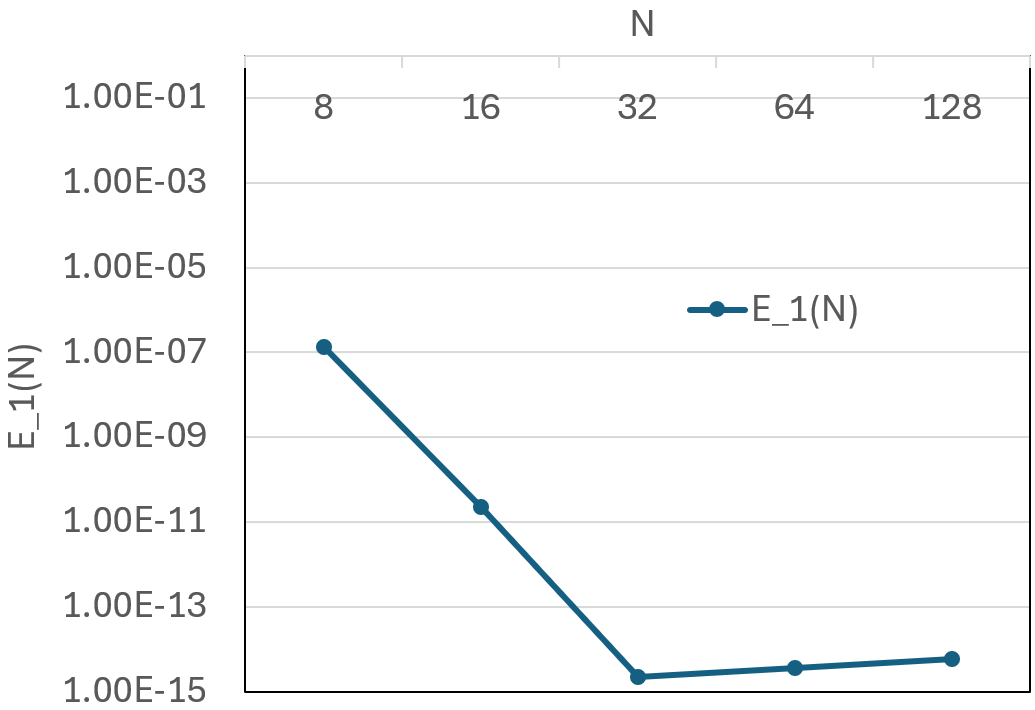}
    \caption{Error of the numerical solution as a function of $N$ for Example 3.}
    \label{fig:Ex3-1}
\end{minipage}
\hspace{3mm}
\begin{minipage}[b]{0.47\columnwidth}
    \centering
    \includegraphics[width=0.9\columnwidth]{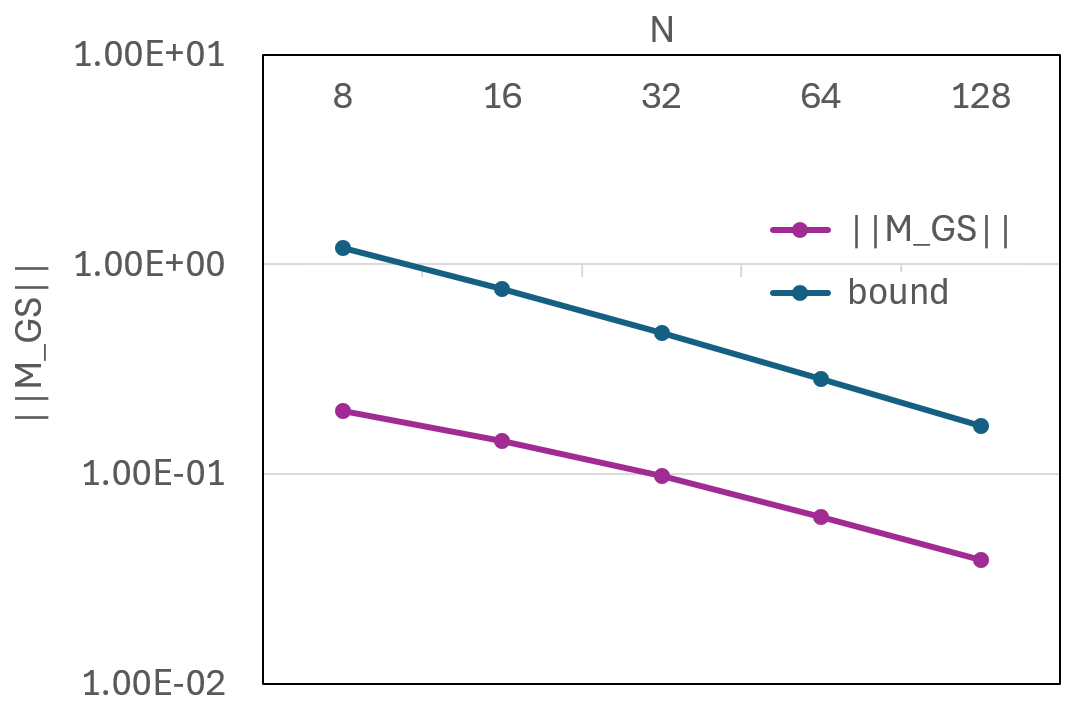}
    \caption{$\|M_{\rm GS}\|_{\infty}$ and its upper bound as a function of $N$ for Example 3.}
    \label{fig:Ex3-2}
\end{minipage}
\end{figure}

\begin{figure}[h]
\centering
\includegraphics[width=0.42\linewidth]{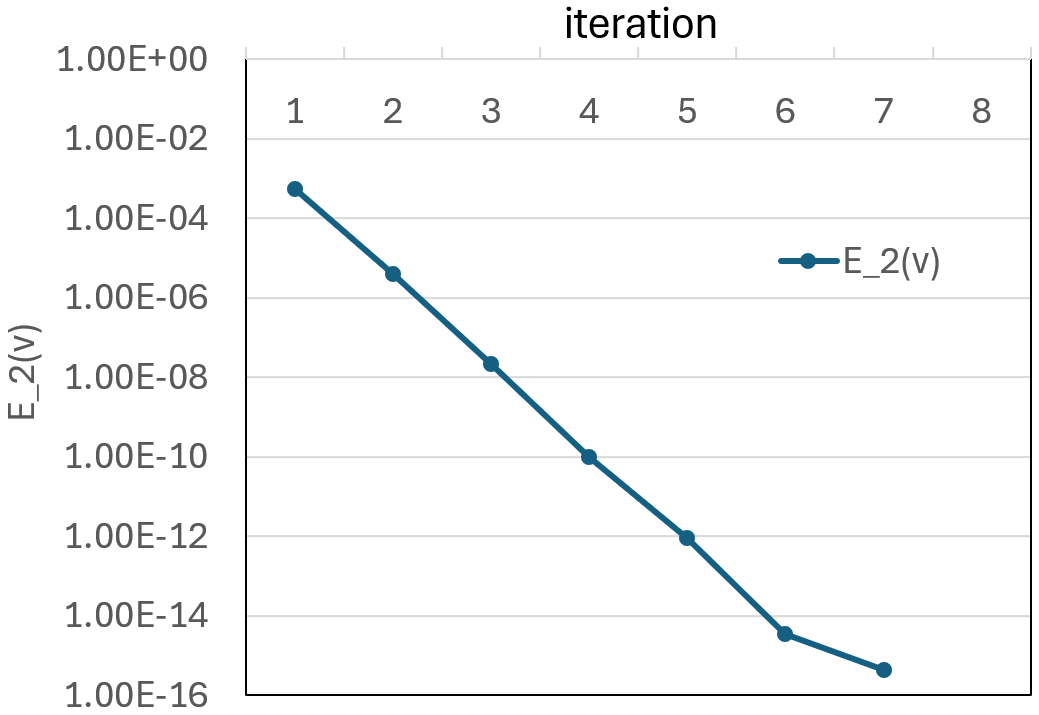}
\caption{Convergence of the Gauss-Seidel type iteration for Example 3.}
\label{fig:Ex3-3}
\end{figure}

\section{Conclusion}
\label{Conclusion}
In this paper, we analyzed the convergence of Gauss-Seidel type iteration for solving nonlinear simultaneous equations arising from the double-exponential Sinc-collocation method for initial value problems. We derived a sufficient condition for convergence of this iteration and also provided an upper bound on its rate of convergence, which explains the rapid convergence of this iteration. Our theoretical results are supported by numerical examples. Our future work includes extending this analysis to other methods such as the IMT-DE-based method recently proposed by Ogata.


\section*{Acknowledgements}
We are grateful to Professor Tomoaki Okayama of Hiroshima City University
and Professor Hidenori Ogata of The University of Electro-Communications
for valuable comments.  In particular, the claim of Lemma~\ref{Lem_diff_DE_monotone} was suggested
by Professor Okayama.

\section*{Funding and/or Conflicts of interests/Competing interests}
This study is partially supported by JSPS KAKENHI Grant Numbers 22K19772, 25H00449, 25K03124, 24K00536, and 24K00540. 
There is no conflict of interest to declare.

\begin{appendices}
\section{Proof of the lemma used in subsection \ref{convergence_rate}}

\begin{lemma}
\label{Lem_diff_DE_monotone}
The function 
\[
f(x) =
\frac{\frac{\pi}{2}\cosh(x)}{\cosh^2\left(\frac{\pi}{2}\sinh(x)\right)}
\]
is even on $\mathbb{R}$ and monotone decreasing on $\{ x \in \mathbb{R} \mid  x \geq 0\}$. 
\end{lemma}

\begin{proof}
It is trivial that $f$ is even. 
Below we prove the monotonicity for $x \geq 0$. 
By letting $u(x) = \frac{\pi}{2} \sinh(x)$, we have
\[
f'(x) = 
\frac{\pi}{2} \cdot
\frac{\sinh(x) \cosh(u(x)) - \pi \cosh^{2}(x) \sinh(u(x))}{\cosh^{3}(u(x))}.
\]
Since the sign of the right-hand side is equal to that of its numerator, 
the inequality $f'(x) \leq 0$ is equivalent to that 
\begin{align}
\frac{\sinh(x)}{\pi \cosh^{2}(x)} \leq \tanh(u(x)). 
\notag
\end{align}
By letting $t = \sinh(x)$,
this inequality is equivalent to that
\begin{align}
\frac{t}{\pi (1+t^{2})} \leq \tanh \left( \frac{\pi}{2} t \right). 
\label{eq:ineq_equiv_concl}
\end{align}
Therefore it suffices to prove this inequality for any $t \geq 0$.

Let $g(t)$ be defined by $g(t) = \frac{t}{\pi (1+t^{2})}$ for $t \geq 0$. 
Then we have
\begin{align}
g'(t) = \frac{1-t^{2}}{\pi(1+t^{2})^{2}}
\quad \text{and} \quad
g''(t) = \frac{2t(-t^{3} + 2t^{2} - t -2)}{\pi(1+t^{2})^{3}}.
\notag
\end{align}
Therefore $g(t)$ attains its maximum value $\frac{1}{2\pi}$ at $t = 1$. 
In addition, 
for $t \in [0, 1]$, 
we have $-t^{3} + 2t^{2} - t -2 \leq -t^{3} + 2 - t - 2 = -t^{3} - t \leq 0$, 
which implies $g''(t) \leq 0$. 
Therefore $g(t)$ is concave on $[0,1]$ and we have
\begin{align}
g(t) \leq \frac{1}{\pi} t \quad (=g'(0)t).
\label{eq:key_ineq_g_t_on_0_1}
\end{align}
for $t \in [0,1]$. 
Next we show 
\begin{align}
\frac{1}{\pi} < \frac{\mathrm{e}^{\pi} - 1}{\mathrm{e}^{\pi} + 1}
\quad 
\left( = \tanh \left( \frac{\pi}{2} \right) \right).
\label{eq:key_ineq_1_pi_tanh}
\end{align}
Indeed, this inequality is equivalent to that 
\(
\mathrm{e}^{\pi} > \frac{\pi + 1}{\pi-1}
\), which is shown as
\[
\mathrm{e}^{\pi}
\geq 1 + 3 + \frac{3^2}{2!} + \frac{3^3}{3!}
= 13 > 2 = 
\frac{3+1}{3-1}
\geq 
\frac{\pi + 1}{\pi-1}. 
\]
Then, 
by combining~\eqref{eq:key_ineq_g_t_on_0_1}, \eqref{eq:key_ineq_1_pi_tanh}, 
and the concavity of $t \mapsto \tanh(\frac{\pi}{2}t)$ on $[0,1]$, 
we have
\begin{align}
g(t) \leq \frac{1}{\pi} t \leq \tanh \left( \frac{\pi}{2} \right) t \leq \tanh \left( \frac{\pi}{2} t \right)
\notag
\end{align}
for any $t \in [0,1]$. 
This implies inequality~\eqref{eq:ineq_equiv_concl} for $t \in [0,1]$.
For $t \geq 1$, $g(t)$ is decreasing and $t \mapsto \tanh(\frac{\pi}{2}t)$ is increasing. 
Therefore inequality~\eqref{eq:ineq_equiv_concl} also holds for  $t \geq 1$. 
Thus we obtain inequality~\eqref{eq:ineq_equiv_concl} for  any $t \geq 0$.
This completes the proof of the lemma.
\end{proof}

\begin{lemma}
\label{LemA-1}
When $x>0$, the following inequality holds.
\begin{equation}
\left|\frac{\pi}{2}-{\rm Si}(x)\right| \le \frac{1}{x}.
\end{equation}
\end{lemma}
\begin{proof}
The proof that follows is a detailed version of the proof given in \cite{stackexchange}. First, recall the following two equalities from \cite[Eqs.~5.2.12 and 5.2.13]{Abramowitz65}.
\begin{align}
\int_0^{\infty}\frac{\sin s}{x+s}\,{\rm d}s &= \int_0^{\infty}\frac{e^{-xs}}{1+s^2}\,{\rm d}s, \\
\int_0^{\infty}\frac{\cos s}{x+s}\,{\rm d}s &= \int_0^{\infty}\frac{se^{-xs}}{1+s^2}\,{\rm d}s.
\end{align}
Using these, we have
\begin{align}
\frac{\pi}{2}-{\rm Si}(x) &= \int_0^{\infty}\frac{\sin t}{t}\,{\rm d}t-\int_0^x\frac{\sin t}{t}\,{\rm d}t \nonumber \\
&= \int_x^{\infty}\frac{\sin t}{t}\,{\rm d}t \nonumber \\
&= \int_0^{\infty}\frac{\sin x\cos s+\cos x\sin s}{x+s}\,{\rm d}s \nonumber \\
&= \int_0^{\infty}\frac{\sin x\cos s+\cos x\sin s}{x+s}\,{\rm d}s \nonumber \\
&= \int_0^{\infty}\frac{s\sin x+\cos x}{1+s^2}\,e^{-xs}\,{\rm d}s.
\end{align}
Taking the absolute values of both sides and applying the Cauchy-Schwarz inequality to the numerator of the integrand in the right-hand side gives
\begin{align}
\left|\frac{\pi}{2}-{\rm Si}(x)\right| &\le \int_0^{\infty}\frac{\sqrt{1+s^2}\sqrt{\sin^2 x+\cos^2 x}}{1+s^2}\,e^{-xs}\,{\rm d}s \nonumber \\
&= \int_0^{\infty}\frac{e^{-xs}}{\sqrt{1+s^2}}\,{\rm d}s \nonumber \\
&\le \int_0^{\infty}e^{-xs}\,{\rm d}s = \frac{1}{x},
\end{align}
which proves the lemma.
\end{proof}

\section{Exact solution of the Lotka-Volterra equation with multiple species}
\label{derivation}
In this appendix, we explain how to derive an exact solution to the Lotka-Volterra equation with $2m-1$ species:
\begin{align}
\begin{dcases}
& \frac{{\rm d}x_k}{{\rm d}t}=x_k(t)( x_{k+1}(t)-x_{k-1}(t)) \q (k=1, 2, \ldots, 2m-1), \\
& x_k(0)=x_k^{(0)} \q (k=1, 2, \ldots, 2m-1), \\
& x_0(t)\equiv 0, \q x_{2m}(t)\equiv 0,
\end{dcases}
\label{eq:LV}
\end{align}
which is used as a test problem in the numerical experiments section. Our strategy is to construct a solution to \eqref{eq:LV} from that of the Toda lattice equation:
\begin{align}
\begin{dcases}
& \frac{{\rm d}q_k}{{\rm d}t} = e_k(t)-e_{k-1}(t) \q (k=1, \ldots, m), \\
& \frac{{\rm d}e_k}{{\rm d}t} = e_k(t)(q_{k+1}(t)-q_k(t)) \q (k=1, \ldots, m-1), \\
& q_k(0)=q_k^{(0)} \q (k=1, \ldots, m), \\
& e_k(t)=e_k^{(0)} \q (k=1, \ldots, m-1), \\
& e_0(t)\equiv 0, \q e_m(t)\equiv 0.
\end{dcases}
\label{eq:Toda}
\end{align}
The following lemma holds.
\begin{lemma}
\label{LemB-1}
Let $\{q_k(t)\}_{k=1}^m$ and $\{e_k(t)\}_{k=1}^{m-1}$ be the solution to the Toda lattice equation \eqref{eq:Toda}. Then, the functions $\{x_k(t)\}_{k=1}^{2m-1}$ defined by
\begin{equation}
\begin{dcases}
& q_k(t)=1+x_{2k-2}(t)+x_{2k-1}(t) \q (k=1, \ldots, m), \\
& e_k(t)=x_{2k-1}(t)x_{2k}(t) \q (k=1, \ldots, m-1), \\
& x_0(t)\equiv 0
\end{dcases}
\label{eq:Miura}
\end{equation}
satisfy the Lotka-Volterra equation \eqref{eq:LV}. The corresponding initial values are obtained from $\{q_k(0)\}_{k=1}^m$ and $\{e_k(0)\}_{k=1}^{m-1}$ by solving \eqref{eq:Miura} for $\{x_k(0)\}_{k=1}^{2m-1}$.
\end{lemma}

\begin{remark}
We note that \eqref{eq:Miura} actually defines $\{x_k(t)\}_{k=1}^{2m-1}$ for generic $\{q_k(t)\}_{k=1}^m$ and $\{e_k(t)\}_{k=1}^{m-1}$. In fact, since $x_0(t)=0$, we have $x_1(t)=q_1(t)-1$ from the first equation of \eqref{eq:Miura}. Now, assume that $\{x_{\ell}(t)\}_{\ell=1}^{2k-1}$ have been determined. Then, if $x_{2k-1}(t)\ne 0$, we have $x_{2k}(t)=e_k(t)/x_{2k-1}(t)$ from the second equation of \eqref{eq:Miura} and then $x_{2k+1}(t)=q_{k+1}(t)-x_{2k}(t)-1$ from the first equation of \eqref{eq:Miura}. Hence, the claim holds by induction. Eq.~\eqref{eq:Miura} is a continuous-time version of the so-called Miura transformation between the Toda lattice equation and the Lotka-Volterra equation \cite{Hirota95}.
\end{remark}
\begin{proof}
For brevity, we denote $x_k(t)$ by $x_k$ and ${\rm d}x_k/{\rm d}t$ by $\dot{x}_k$. First, by substituting \eqref{eq:Miura} into the first equation of \eqref{eq:Toda} with $k=1$ and noting that $x_0(t)=0$ and $e_0(t)=0$, we obtain
\begin{equation}
\dot{x}_1=x_1 x_2,
\end{equation}
which is the first equation of \eqref{eq:LV} with $k=1$. Now, assume that $\{x_{\ell}(t)\}_{\ell=1}^{2k-1}$ have been shown to satisfy the first equation of \eqref{eq:LV}. Then, substituting \eqref{eq:Miura} into the second equation of \eqref{eq:Toda}, we have
\begin{equation}
\frac{\rm d}{{\rm d}t}(x_{2k-1}x_{2k})=x_{2k-1}x_{2k}\left\{(1+x_{2k}+x_{2k+1})-(1+x_{2k-2}+x_{2k-1})\right\},
\end{equation}
or
\begin{equation}
\dot{x}_{2k-1}x_{2k}+ x_{2k-1}\dot{x}_{2k}=x_{2k-1}x_{2k}(x_{2k}+x_{2k+1}-x_{2k-2}-x_{2k-1})
\end{equation}
Using $\dot{x}_{2k-1}=x_{2k-1}(x_{2k}-x_{2k-2})$, we can rewrite this as
\begin{equation}
x_{2k-1}(x_{2k}-x_{2k-2})x_{2k}+ x_{2k-1}\dot{x}_{2k}=x_{2k-1}x_{2k}(x_{2k}+x_{2k+1}-x_{2k-2}-x_{2k-1}),
\end{equation}
or
\begin{equation}
x_{2k-1}\dot{x}_{2k}=x_{2k-1}x_{2k}(x_{2k+1}-x_{2k-1})
\end{equation}
If $x_{2k-1}(t)\ne 0$ (note that this is a necessary condition for the transformation \eqref{eq:Miura} to be well-defined), we can divide both sides by $x_{2k-1}$ to obtain
\begin{equation}
\dot{x}_{2k}=x_{2k}(x_{2k+1}-x_{2k-1}), \label{eq:u2kdot}
\end{equation}
which is the first equation of \eqref{eq:LV} for $2k$. Further, substituting \eqref{eq:Miura} into the first equation of \eqref{eq:Toda} for $q_{k+1}$ gives
\begin{equation}
\frac{{\rm d}}{{\rm d}t}(1+x_{2k}+x_{2k+1})=x_{2k+1}x_{2k+2}-x_{2k-1}x_{2k}.
\end{equation}
Substituting \eqref{eq:u2kdot} into the left-hand side yields
\begin{equation}
\dot{x}_{2k+1}=x_{2k+1}(x_{2k+2}-x_{2k}),
\end{equation}
which is the first equation of \eqref{eq:LV} for $2k+1$. Hence, we have shown that $\{x_{\ell}(t)\}_{\ell=1}^{2k+1}$ satisfy the first equation of \eqref{eq:LV}. Continuing this process until $k=m-1$, we can show that $\{x_{\ell}(t)\}_{\ell=1}^{2m-1}$ satisfy the Lotka-Volterra equation. The statement about the initial values follows naturally because $\{q_k(t)\}_{k=1}^m$ and $\{e_k(t)\}_{k=1}^{m-1}$ are connected with $\{x_k(t)\}_{k=1}^{2m-1}$ through \eqref{eq:Miura} for any $t$.
\end{proof}

Next, we describe how to solve the initial value problem of the Toda lattice equation \eqref{eq:Toda}. Let us define an $m\times m$ tridiagonal matrix $A=A(t)$ and its strictly lower triangular part $A_-=A_-(t)$ by
\begin{equation}
A=\left(\begin{array}{cccc}
q_1 & 1 & & \\
e_1 & q_2 & \ddots & \\
& \ddots & \ddots & 1 \\
& & e_{m-1} & q_m
\end{array}
\right), \q
A_-=\left(\begin{array}{cccc}
& & & \\
e_1 & & & \\
& \ddots & & \\
& & e_{m-1} &
\end{array}
\right).
\end{equation}
Then, it can easily be verified that the matrix differential equation
\begin{equation}
\frac{{\rm d}A}{{\rm d}t} =[A,A_-],
\label{eq:Lax}
\end{equation}
where $[A,B]=AB-BA$, is equivalent to the Toda lattice equation \eqref{eq:Toda}. More specifically, the differential equations for $q_k(t)$ arise from the equalities among the diagonal elements of \eqref{eq:Lax}, while those for $e_k(t)$ arise from the equalities among the lower subdiagonal elements of \eqref{eq:Lax}. Eq.~\eqref{eq:Lax} is called the Lax form of \eqref{eq:Toda}.

Once expressed in the form of \eqref{eq:Lax}, the Toda lattice equation can be solved by the group theoretic method \cite{Kostant79,Nakamura18}. The procedure to compute $\{q_k(t)\}_{k=1}^m$ and $\{e_k(t)\}_{k=1}^{m-1}$ for given $t$ is as follows.
\begin{itemize}
\item[(i)] Construct $A(0)$ from $\{q_k(0)\}_{k=1}^m$ and $\{e_k(0)\}_{k=1}^{m-1}$.
\item[(ii)] Compute $\exp(tA(0))$ and its $LR$ decomposition $\exp(tA(0))=L_t R_t$, where $L_t$ is a unit lower subdiagonal matrix and $R_t$ is an upper bidiagonal matrix.
\item[(iii)] Let $A(t)\equiv L_t^{-1}A(0)L_t$.
\item[(iv)] Obtain $\{q_k(t)\}_{k=1}^m$ and $\{e_k(t)\}_{k=1}^{m-1}$ as the diagonal elements and the lower subdiagonal elements of $A(t)$, respectively.
\end{itemize}
By combining this procedure with Lemma \ref{LemB-1}, we can construct an exact solution to the Lotka-Volterra equation.

For general $m$ and $A(0)$, the matrix exponential $\exp(tA(0))$ has to be calculated numerically. However, when $m=2$, everything can be computed analytically. Below, we will give a sample of such cases.

\paragraph{Example 1}
Let us consider the case of $m=2$, $q_1^{(0)}=q_2^{(0)}=3$ and $e_1^{(0)}=1$. Then, we have
\begin{equation}
A(0)=\left(\begin{array}{cc}
3 & 1 \\
1 & 3
\end{array}
\right)
=\left(\begin{array}{cc}
\frac{1}{\sqrt{2}} & \frac{1}{\sqrt{2}} \\
\frac{1}{\sqrt{2}} & -\frac{1}{\sqrt{2}}
\end{array}
\right)
\left(\begin{array}{cc}
4 & 0 \\
0 & 2
\end{array}
\right)
\left(\begin{array}{cc}
\frac{1}{\sqrt{2}} & \frac{1}{\sqrt{2}} \\
\frac{1}{\sqrt{2}} & -\frac{1}{\sqrt{2}}
\end{array}
\right)^{-1}
\end{equation}
and
\begin{equation}
\exp(tA(0))= \left(\begin{array}{cc}
\frac{1}{\sqrt{2}} & \frac{1}{\sqrt{2}} \\
\frac{1}{\sqrt{2}} & -\frac{1}{\sqrt{2}}
\end{array}
\right)
\left(\begin{array}{cc}
e^{4t} & 0 \\
0 & e^{2t}
\end{array}
\right)
\left(\begin{array}{cc}
\frac{1}{\sqrt{2}} & \frac{1}{\sqrt{2}} \\
\frac{1}{\sqrt{2}} & -\frac{1}{\sqrt{2}}
\end{array}
\right)^{-1}
=\frac{1}{2}\left(\begin{array}{cc}
e^{4t}+e^{2t} & e^{4t}-e^{2t} \\
e^{4t}-e^{2t} & e^{4t}+e^{2t}
\end{array}
\right).
\end{equation}
Thus,
\begin{equation}
\exp(tA(0))
=\left(\begin{array}{cc}
1 & 0 \\
\tanh t & 1
\end{array}
\right)
\left(\begin{array}{cc}
\frac{e^{4t}+e^{2t}}{2} & \frac{e^{4t}-e^{2t}}{2} \\
0 & \frac{e^{4t}+e^{2t}}{2}-\frac{(e^{4t}-e^{2t})^2}{2(e^{4t}+e^{2t})}
\end{array}
\right)=L_t R_t
\end{equation}
and
\begin{align}
A(t) = L_t^{-1}A(0)L_t
&= \left(\begin{array}{cc}
1 & 0 \\
-\tanh t & 1
\end{array}
\right)
\left(\begin{array}{cc}
3 & 1 \\
1 & 3
\end{array}
\right)
\left(\begin{array}{cc}
1 & 0 \\
\tanh t & 1
\end{array}
\right) \nonumber \\
&= \left(\begin{array}{cc}
3+\tanh t & 1 \\
\frac{1}{\cosh^2 t} & 3-\tanh t
\end{array}
\right).
\end{align}
This corresponds to the solution:
\begin{equation}
q_1(t)= 3+\tanh t, \q q_2(t)= 3-\tanh t, \q e_1(t)=\frac{1}{\cosh^2 t}.
\end{equation}
It is easy to verify that these functions actually satisfy the Toda lattice equation \eqref{eq:Toda}.

Furthermore, by invoking Lemma \ref{LemB-1}, we can construct the corresponding solution of the Lotka-Volterra equation as follows.
\begin{align}
x_1(t) &= q_1(t)-1 = 2+\tanh t, \label{eq:LV3sol1} \\
x_2(t) &= \frac{e_1(t)}{u_1(t)} = \frac{1}{\cosh t(2\cosh t+\sinh t)}, \label{eq:LV3sol2} \\
x_3(t) &= q_2(t)-1-u_2(t)=2-\tanh t-\frac{1}{\cosh t(2\cosh t+\sinh t)}. \label{eq:LV3sol3}
\end{align}
This is a solution corresponding to the initial condition $x_1^{(0)}=2$, $x_2^{(0)}=\frac{1}{2}$ and $x_3^{(0)}=\frac{3}{2}$. We used this solution in Section \ref{Results}.

We will also calculate the constants $M$ and $L$ for use in the numerical experiments. Let us assume that $\frac{1}{2}\le \rho\le \frac{3}{2}$ for simplicity.  The region $\mathcal{D}_{\rho}$ is given as
\begin{equation}
\mathcal{D}_{\rho}=\left\{(t,{\bm x})\,\left|\,t\in[a,b], x_1\in\left[2-\rho, 2+\rho\right], x_2\in\left[\frac{1}{2}-\rho, \frac{1}{2}+\rho\right], x_3\in\left[\frac{3}{2}-\rho, \frac{3}{2}+\rho\right]\right.\right\}.
\end{equation}
Then it is easy to see that $M=(2+\rho)(\frac{1}{2}+\rho)$, which is attained when $x_1=2+\rho$ and $x_2=\frac{1}{2}+\rho$. Now, let ${\bm x}=(x_1,x_2,x_3)^{\top}$ and ${\bm x}'=(x_1^{\prime},x_2^{\prime},x_3^{\prime})^{\top}$. Then,
\begin{align}
|f_1(t,{\bm x})-f_1(t,{\bm x}')| &= |x_1 x_2-x_1^{\prime}x_2^{\prime}| \le |x_1|\,|x_2-x_2^{\prime}| + |x_2^{\prime}|\,|x_1-x_1^{\prime}| \nonumber \\
&\le (|x_1|+|x_2^{\prime}|)\|{\bm x}-{\bm x}'\|_{\infty} \le \left(\frac{5}{2}+2\rho\right)\|{\bm x}-{\bm x}'\|_{\infty}.
\end{align}
Similarly,
\begin{align}
|f_2(t,{\bm x})-f_2(t,{\bm x}')| &\le (||x_2|+|x_3^{\prime}-x_1|+|x_2^{\prime}|)\|{\bm x}-{\bm x}'\|_{\infty} \le \left(\frac{3}{2}+4\rho\right)\|{\bm x}-{\bm x}'\|_{\infty}, \\
|f_3(t,{\bm x})-f_3(t,{\bm x}')| &\le (|x_2|+|x_3^{\prime}|)\|{\bm x}-{\bm x}'\|_{\infty} \le (2+2\rho)\|{\bm x}-{\bm x}'\|_{\infty}
\end{align}
and therefore
\begin{equation}
\|{\bm f}(t,{\bm x})-{\bm f}(t,{\bm x}')\|_{\infty} \le \left(\frac{3}{2}+4\rho\right)\|{\bm x}-{\bm x}'\|_{\infty}.
\end{equation}
Hence, we can choose $L=\frac{3}{2}+4\rho$. Note that this is the smallest possible value of $L$ in $\mathcal{D}_{\rho}$. To see this, let
\begin{equation}
x_1=2+\rho, \quad x_2=\frac{1}{2}+\rho, \quad x_3=\frac{3}{2}-\rho, \quad x_1^{\prime}=x_1-\epsilon, \quad x_2^{\prime}=x_2-\epsilon, \quad x_3^{\prime}=x_3+\epsilon,
\end{equation}
where $\epsilon>0$ is a small number.
\end{appendices}

\end{document}